\documentclass[12pt,a4paper]{amsart}
\usepackage{tikz}
\usepackage{pdflscape}
\usepackage{amsmath}
\usepackage{amsfonts}
\usepackage{amssymb}
\usepackage{mathtools}
\usepackage[all]{xy}
\usepackage{color}
\usetikzlibrary{arrows}
\usepackage{float}
\usepackage[shortlabels]{enumitem}
\usepackage{ifpdf}
\ifpdf
  \usepackage[colorlinks,
            linkcolor=magenta,       
            anchorcolor=magenta,     
            citecolor=magenta,       
            final,
            hyperindex]{hyperref}
\else
  \usepackage[colorlinks,
            linkcolor=magenta,       
            anchorcolor=magenta,     
            citecolor=magenta,       
            final,
            hyperindex,driverfallback=hypertex]{hyperref}
\fi
\newtheorem{theorem}{Theorem}[section]
\newtheorem{lemma}[theorem]{Lemma}
\newtheorem{coro}[theorem]{Corollary}

\newtheorem{prop}[theorem]{Proposition}
\theoremstyle{definition}
\newtheorem{defn}[theorem]{Definition}
\newtheorem{remark}[theorem]{Remark}
\newtheorem{exam}[theorem]{Example}

\newcommand{\delete}[1]{}

\newcommand {\Ker}{{\rm Ker\,}}
\newcommand {\Ima}{{\rm Im\,}}
\newcommand {\End}{{\rm End\,}}
\newcommand {\Aut}{{\rm Aut\,}}

\begin{document}
\title[Post  Clifford semigroups and the Yang-Baxter equation]{Post Clifford semigroups, the Yang-Baxter equation, relative Rota--Baxter Clifford semigroups and dual weak left braces}

\author{Xiaoqian Gong}
\address{School of Mathematics, Yunnan Normal University, Kunming, Yunnan 650500, China}
\email{3417259364@qq.com}

\author{Shoufeng Wang$^{\ast}$}\thanks{*Corresponding author}
\address{School of Mathematics, Yunnan Normal University, Kunming, Yunnan 650500, China}
\email{wsf1004@163.com}


\begin{abstract}As generalizations of Rota--Baxter groups,  Rota--Baxter Clifford semigroups have been  introduced  by Catino,  Mazzotta and Stefanelli in 2023. Based on their pioneering results, in this paper we first continue to study Rota--Baxter Clifford semigroups.
Inspired by the corresponding results in Rota--Baxter groups, we firstly obtain some properties and construction methods for Rota--Baxter Clifford semigroups, and then study the substructures and quotient structures of these semigroups.  On the other hand, as generalizations of post-groups, Rota--Baxter Clifford semigroups and braided groups, in this paper we introduce and investigate post Clifford semigroups, relative Rota--Baxter Clifford semigroups and braided Clifford semigroups, respectively. We prove  that the categories of strong post Clifford semigroups, dual weak left braces, bijective strong relative Rota-Baxter Clifford semigroups and braided Clifford semigroups are mutually pairwise equivalent, and  the category  of post Clifford semigroups  is equivalent to the category  of  bijective  relative Rota--Baxter Clifford semigroups, respectively. As a consequence,  we prove that both post Clifford semigroups, relative Rota--Baxter Clifford semigroups and braided Clifford semigroups can provide set-theoretical solutions for the Yang--Baxter equation. The substructures and quotient structures of relative Rota--Baxter Clifford semigroups are also considered.
\end{abstract}
\makeatletter
\@namedef{subjclassname@2020}{\textup{2020} Mathematics Subject Classification}
\makeatother

\maketitle


\tableofcontents

\vspace{-0.8cm}

\section{Introduction}

The Yang-Baxter equation first appeared in Yang's study of exact solutions of many-body problems in \cite{Yang} and Baxter's work of the eight-vertex model in \cite{Baxter}. The equation is closely related to many fields in mathematics and physics. Due to the complexity in solving the Yang--Baxter
equation, Drinfeld suggested to study set-theoretical solutions of the Yang--Baxter equation in \cite{Drinfeld}.  Earlier studies of set-theoretical solutions of the Yang--Baxter equation  can be found in \cite{Etingof-Schedler-Soloviev,Gateva-Ivanova-Van den Bergh,Lu-Yan-Zhu,Soloviev} and the review article \cite{Takeuchi}.

To construct set-theoretical solutions of the Yang--Baxter equation by the view of algebra, several kinds of algebraic structures have been introduced in literature,
see for example \cite{Catino-Colazzo-Stefanelli,Catino-Colazzo-Stefanelli3,Catino-Mazzotta-Miccoli-Stefanelli,Catino-Mazzotta-Stefanelli,Cedo-Jespers-Okninski,
Colazzo-Jespers-Van Antwerpen-Verwimp,Guarnieri-Vendramin,Miccoli1,Liu-Wang,Miccoli2,Rump1,Rump2} and the  references therein.
In particular,  Rump introduced left braces in \cite{Rump2} as a generalization of Jacobson radical rings, and a few years later, Ced$\acute{\rm o}$, Jespers, and Okni$\acute{\rm n}$ski \cite{Cedo-Jespers-Okninski} reformulated
Rump's definition of left braces.   Guarnieri and Vendramin \cite{Guarnieri-Vendramin} introduced skew left braces as generalization of left braces.
Observe that (skew) left braces can provide (involutive) non-degenerate  set-theoretical solutions of the Yang--Baxter equation. Recently, Catino-Mazzotta-Miccoli-Stefanelli \cite{Catino-Mazzotta-Miccoli-Stefanelli} and Catino-Mazzotta-Stefanelli \cite{c1} introduced weak left braces and dual weak left braces by using Clifford semigroups, respectively, and proved that weak left braces and dual weak left braces can provide ceratin special degenerate  set-theoretical solutions of the Yang--Baxter equation. More recent information on dual weak left braces  can be found in  \cite{Catino-Mazzotta-Stefanelli3} and \cite{Mazzotta}.

In 2021 Guo, Lang and Sheng \cite{g2} introduced the notion of Rota--Baxter groups and gave some basic examples and properties of these groups.  Soon after the intial breakthrough in \cite{g2}, quite a few studies were carried out in this direction (see \cite{b2,b1,Caranti,c1,d1,Galt,Gao1,Gao2,Gao3,l1,Rathee}). In particular, Bardakov and Gubarev \cite{b2} have investigated the relationship between skew left braces and  Rota--Baxter groups, and shown that every Rota--Baxter group gives rise to a skew left brace and every left brace can be  embedded into a Rota--Baxter group.   In 2023, Catino,  Mazzotta and Stefanelli \cite{c1} defined Rota--Baxter operators for Clifford semigroups, extended some results in \cite{b2} to Clifford semigroups and connected Rota--Baxter Clifford semigroups with dual weak left  braces.
On the other hand, in 2024 Jiang,  Sheng and Zhu \cite{Jiang} established a local Lie theory for relative Rota--Baxter operators of weight 1 and generalized the idea
of Rota-Baxter operators to relative Rota-Baxter operators on Lie groups. Since then, many authors have devoted themselves to the study of relative   Rota--Baxter groups (see \cite{Bardakov1,Bardakov2,Belwal1,Belwal2,Belwal3,l2}). In particular, Rathee and Singh \cite{l2} have explored the connections of relative Rota--Baxter groups
with skew left braces and shown that the category of bijective relative Rota--Baxter groups is equivalent to the category of skew left braces.

Motivated by the notion of post-Lie algebras, Bai, Guo, Sheng and Tang have put forward  the notions of post-groups recently in \cite{Bai} by which relative Rota--Baxter groups, braided groups and skew left braces are tied together. The defining axioms of  post-groups have been relaxed by S. Wang in \cite{Wang}, leading to the notion of weak twisted post-group. More recently,   Al-Kaabi, Ebrahimi-Fard and Manchon \cite{Al-Kaabi-Ebrahimi-Fard-Manchon} considered free post-groups and post-groups from group actions.

In this paper, motivated by  the  results given in \cite{Bai,b1,c1,Jiang,l2}, we consider more properties of Rota--Baxter Clifford semigroups and initiate the study of post Clifford semigroups, relative Rota--Baxter Clifford semigroups and braided Clifford semigroups, and associate these Clifford semigroups with dual weak left braces and the Yang--Baxter equation. In particular, we show that both post Clifford semigroups, relative Rota--Baxter Clifford semigroups and braided Clifford semigroups can provide set-theoretical solutions for the Yang--Baxter equation.

The paper is organized as follows. After giving some preliminaries on Clifford semigroups, Rota-Baxter Clifford semigroups and dual weak left braces in Section 2,
we consider  Rota--Baxter Clifford semigroups in Section 3. More specifically, we give some construction methods and study the substructures and quotient structures for Rota--Baxter Clifford semigroups. In section 4 we introduce the notion of post Clifford semigroups, prove that the category  of strong post Clifford semigroups  is isomorphic to the category of dual weak left braces. Section 5 introduces the notion of relative Rota--Baxter Clifford semigroups and shows that the category of bijective  relative Rota-Baxter Clifford semigroups is equivalent to the category  of post Clifford semigroups and the category  of bijective strong  relative Rota--Baxter Clifford semigroups is equivalent to the category  of strong post Clifford semigroups, respectively. Moreover, substructures and quotient structures  for relative Rota--Baxter Clifford semigroups are also considered. The final section is devoted to braided Clifford semigroups. Among other things, we prove that the category of strong post Clifford semigroups  is isomorphic to the category of braided Clifford semigroups.  The results obtained in the present paper can be regarded as the enrichment and extension of the corresponding results in  the texts \cite{Bai,b1,c1,Jiang,l2}.

\section{Preliminaries}
In this section, we recall some notions and known results on Clifford semigroups, Rota--Baxter Clifford semigroups and dual weak left braces.
Let $(S, +)$ be a semigroup. As usual,  we denote the set of all idempotents in $(S, +)$ by $E(S,+)$ and the center of $S$ by $C(S,+)$, respectively.  That is,
$$E(S, +)=\{x\in S\mid x+x=x\},\,\, C(S,+)=\{x\in S\mid x+a=a+x \mbox{ for all } a\in S\}.$$
According to \cite{A10}, a semigroup $(S,+)$ is called an {\em inverse semigroup} if for all $x\in S$, there exists a unique element $-a\in S$ such that $a=a-a+a$ and $-a+a-a=-a$. Such an element $-a$ is called {\em the  von Neumann inverse} of $a$ in $S$. Let $(S, +)$ be an inverse semigroup. It is well known that $-(a+b)=-b-a$ and $-(-a)=a$ for all $a,b\in S$. Moreover, $E(S,+)$ is a commutative subsemigroup of $(S, +)$, and $$E(S, +)=\{-x+x\mid x\in S\}=\{x-x\mid x\in S\} \mbox{ and }-e=e \mbox{ for all }e\in E(S,+).$$

 The following lemma  gives a characterization of inverse semigroups. Recall that a semigroup $(S,+)$ is {\em regular} if or each $a\in S$, there exists $a'$ such that $a+a'+a=a$.
\begin{lemma}[\cite{A10}]\label{inverse} A semigroup $(S,+)$ is an inverse semigroup if and only if $(S,+)$ is  regular and $e+f=f+e$ for all $e,f\in E(S, +)$.
\end{lemma}

An inverse semigroup $(S, +)$ is called a {\em Clifford semigroup} if $-a+a=a-a$ for all $a\in S$.
In this case, we denote $a^{0}=-a+a=a-a$ for all $a\in S$. To give some characterizations of Clifford semigroups, we need some additional notations.
Recall that a {\em lower semilattice} $(Y, \leq)$ is a partial ordered set  in which any two elements $\alpha$ and $\beta$ have a greatest lower bound $\alpha \wedge \beta$. Let $(Y, \leq)$ be a lower semilattice and $(G_{\alpha},+) (\alpha\in Y)$ be a family of groups. Assume that for all $\alpha,\beta\in Y$ with $\alpha\geq\beta$, there exists a group homomorphism $\varphi_{\alpha,\beta}: (G_{\alpha},+)\rightarrow (G_{\beta},+)$ such that the following conditions hold:
\begin{itemize}
\item  $\varphi_{\alpha,\alpha}$ is the identity isomorphism on $(G_{\alpha},+)$ for each $\alpha\in Y$.
\item $\varphi_{\beta,\gamma}\varphi_{\alpha,\beta}=\varphi_{\alpha,\gamma}$ for all $\alpha,\beta,\gamma\in Y$ with $\alpha\geq \beta \geq \gamma$.
\end{itemize}
Define a multiplication on $S=\underset{\alpha\in Y}{\bigcup}G_{\alpha}$ as follows:
$$a+b=\varphi_{\alpha,\alpha\wedge \beta}(a)+\varphi_{\beta,\alpha\wedge \beta}(b) \mbox{ for all } a\in G_{\alpha} \mbox{ and } b\in G_{\beta}.$$
Then $S$ forms a Clifford semigroup with respect to the above addition  and is called {\em the strong semilattice of the family of groups $(G_{\alpha},+)$ with respect to the family of group homomorphisms $\varphi_{\alpha,\beta}$}. In this case, we denote  $S=(Y,G_{\alpha},\varphi_{\alpha,\beta})$.
The following lemma  gives some characterizations of Clifford semigroups.

\begin{lemma}[\cite{A10,lawson}]\label{Clifford}For a semigroup $(S, +)$,  the following statements are equivalent:
\begin{itemize}
\item[(1)] $(S,+)$ is a  Clifford semigroup
\item[(2)]$(S,+)$ is a strong semilattice of a family of groups.
\item[(3)] $(S,+)$ is a regular semigroup and $E(S, +) \subseteq C(S,+)$.
\end{itemize}
\end{lemma}
Let $(S, +)$ and $(S', +')$ be  Clifford semigroups. A map (respectively, bijection) $\phi: S\rightarrow S'$ is called a
{\em   semigroup homomorphism} (respectively, {\em   semigroup isomorphism}) from $(S, +)$ to $(S', +')$ if $\phi(x+y)=\phi(x)+'\phi(y)$ for all $x,y\in S$.
A {\em   semigroup homomorphism} (respectively,  {\em   semigroup isomorphism}) from $(S, +)$ to itself is called a {\em   semigroup endomorphism} (respectively,  {\em   semigroup automorphism}) on $(S, +)$. The set of all {\em   semigroup endomorphisms} (respectively,  {\em   semigroup automorphisms}) on $(S, +)$ are denoted by $\End (S,+)$  and $\Aut (S,+)$, respectively. Obviously, $\End (S,+)$  and $\Aut (S,+)$ forms  a semigroup and a group with respect to the compositions of maps, respectively.
In the sequel we shall use the following lemma throughout the paper without further mention.
\begin{lemma}\label{basicproperties}\label{jichu} Let $(S, +)$ and $(S',+')$ be two Clifford semigroups, $a,b\in S$, $e\in E(S, +)$ and $\phi: S\rightarrow S'$ be a semigroup homomorphism. Then
$$a=a-a+a,\,\, -a+a-a=-a,\,\,\, -(a+b)=-b-a,\,\, -(-a)=a,\,\, e+a=a+e,  $$$$e^0=e=-e,\,\,(-a)^0=-a^0=a^{00}=a^0=-a+a=a-a,\,\,a^0+a=a+a^0=a,$$$$ a^0 +b=b+a^0, (a+b)^0=a^0+b^0,\,\ \phi(-a)=-\phi(a),\,\,\, \phi(a^0)=\phi(a)^0.$$
\end{lemma}
Now, we state the notion of  Rota--Baxter operators on   Clifford semigroups which is introduced by Catino Mazzotta and Stefanelli \cite{c1} as a generalization of the notion of Rota--Baxter operators on groups introduced by  Guo, Lang and Sheng in \cite{g2}.

\begin{defn}[\cite{c1}]\label{luobasuanzi} Let $(S,+)$ be a  Clifford semigroup and $R: S\rightarrow S$ be a map. Then $R$ is called a {\em Rota--Baxter operator on $(S,+)$ (with weight $1$)}
if  $$R(a)+R(b)=R(a+R(a)+b-R(a)),\, a+R(a)^{\circ}=a \mbox{ for all } a,b\in S.$$ In this case, we also say that $(S, +, R)$  is a {\em Rota--Baxter Clifford semigroup}.
\end{defn}

\begin{lemma}\label{cwang}  Let $(S,+,R)$ be a  Rota--Baxter  Clifford semigroup and $a,b\in S$. The the following statements are true:
\begin{itemize}
\item[(1)] $R(a^{0})=R(a)^{0}$.
\item[(2)] $R(a)+R(-a)=R(a+R(a)-a-R(a))$.
\item[(3)] $R(a)+R(R(a))=R(a+R(a))$.
\item[(4)] $-R(a)=R(-R(a)-a+R(a))$.
\item[(5)]$a^{0}+R(a)^{0}=a^{0}$.
\item[(6)]  $-R(a)+R(R(a)+b-R(a))=R(-R(a)-a+R(a)+b)$.
\end{itemize}
\end{lemma}
\begin{proof}Items (1)--(4) follow from Proposition 10 and Remark in {\cite{c1}}.  By Definition \ref{luobasuanzi},   $$a^{0}=(a+R(a)^{0})^{0}=a^{0}+R(a)^{00}=a^{0}+R(a)^{0}.$$ This gives (5).  Finally, by (4) and  Definition \ref{luobasuanzi},
$$-R(a)+R(R(a)+b-R(a))=R(-R(a)-a+R(a))+R(R(a)+b-R(a))$$$$=R(-R(a)-a+R(a)-R(a)+R(a)+b-R(a)+R(a))
=R(-R(a)-a+R(a)+b).$$ Thus (6) holds.
\end{proof}
By \cite[Theorem 8]{c1}, we can formulate the following definition.
\begin{defn}
Let $(S,+)$ be a Clifford semigroup, where $S=(Y,G_{\alpha},\varphi_{\alpha,\beta})$.  A Rota--Baxter operator $R$  on $(S,+)$  is called {\em strong} if $R(G_\alpha)\subseteq G_\alpha$ for all $\alpha\in Y$.
\end{defn}
\begin{lemma} \label{strong}Let $R$ be a Rota--Baxter operator on a Clifford semigroup $(S, +)$, where $S=(Y,G_{\alpha},\varphi_{\alpha,\beta})$. Then $R$ is strong if and only if $R(e)=e$ for all $e\in E(S,+)$.
\end{lemma}
\begin{proof}Let $e\in E(S,+)$. Then $e\in G_{\alpha}$ for some $\alpha\in Y$. By Lemma \ref{cwang} (1), we have $R(e)\in E(S,+)$. Since $R$ is strong, it follows that
$R(e)\in E(S,+)\cap G_{\alpha}$. Observe that $G_{\alpha}$ is a group, $G_{\alpha}$ contains exactly one idempotent $e$. This shows that $R(e)=e$.
Conversely, let $R(e)=e$ for all $e\in E(S,+)$. Assume that $a\in G_{\alpha},\alpha\in Y$. Then $a^{0}\in G_{\alpha}$. Let $R(a)\in G_{\beta},\beta\in Y$.  Then $R(a)^{0}\in G_{\beta}$. By Lemma \ref{cwang}, we have $G_{\beta}\ni R(a)^{0}=R(a^{0})=a^{0}\in G_{\alpha},$
as $R$ is strong and $a^{0}\in E(S,+)$. This implies that $\alpha=\beta$ and so $R(a)\in G_{\alpha}$. Thus $R(G_{\alpha})\subseteq G_{\alpha}$ for all $\alpha\in Y$. So $R$ is strong.
\end{proof}
\begin{defn}
Let $(S, +, R)$ and $(S', +',  R')$ be two Rota--Baxter Clifford semigroups. A map $\phi: S\rightarrow S'$ is called a
{\em  Rota--Baxter homomorphism} if $\phi$ is a semigroup homomorphism such that $\phi R=R'\phi$.
\end{defn}
According to \cite{c1}, Rota--Baxter Clifford semigroups are  closely related to so-called {\em dual weak  left braces}  closely.
\begin{defn}[\cite{c1}]  A triple $(S, +, \circ)$ is called a {\em dual weak  left brace} if $(S, +)$ and $(S, \circ)$ are  Clifford  semigroups and the following axioms hold:
 $$
x\circ(y+z)=x\circ y-x+x\circ z,\,\,\,\,\,\, x\circ x^{-1}=-x+x.$$
By the second axiom, we have $E(S, +)=E(S, \circ)$.  We  denote this set by $E(S)$ sometimes.
\end{defn}
For dual weak  left braces, we have the following key lemmas.
\begin{lemma}[Lemma 2 in \cite{Catino-Mazzotta-Stefanelli3}]\label{wang5}
Let $(T, +, \circ)$ be a dual weak left  brace.  Then $$e\circ a=a\circ e=e+a=a+e \mbox{ for all } a\in S \mbox{ and } e\in E(T).$$
\end{lemma}
\begin{lemma}[\cite{Catino-Mazzotta-Miccoli-Stefanelli,Catino-Mazzotta-Stefanelli}]\label{wangkangcccc}
Let $(T, +, \circ)$ be a dual weak left  brace. Define $$\lambda^T_a: T\rightarrow T, b\mapsto -a+a\circ b$$ for all $a\in T$.  Then $\lambda^T_a\in \End (T,+)$.
Furthermore, define $$\lambda^T: (T, \circ)\rightarrow \End (T,+), \,\, a\mapsto \lambda^T_a.$$ Then $\lambda^T$ is a semigroup homomorphism.
\end{lemma}
Let $(S, +,R)$ be a Rota--Baxter Clifford semigroup. Define a new binary operation $\circ_R$ on $S$ as follows:
For all $a,b\in S$, $a\circ_R b=a+R(a)+b-R(a).$ Then by \cite[Lemma 15, Remark 3 and Theorem 16]{c1}, we have the following lemma.
\begin{lemma}[\cite{c1}]\label{wang9}Let $(S, +,R)$ be a Rota--Baxter Clifford semigroup.
\begin{itemize}
\item[\rm (1)] $(S, +, \circ_R)$ forms a dual weak  left brace such that $E(S, +)=E(S, \circ_R)$ and the inverse   of $a$ in $(S, \circ_R)$ is $a^{-1}=-R(a)-a+R(a)$.
\item[\rm (2)] $(S, \circ_R, R)$ is also a Rota--Baxter  Clifford semigroup.
\item[\rm (3)] $R$ is a Rota--Baxter homomorphism  from $(S, \circ_R, R)$ to $(S, +, R)$.
\end{itemize}
\end{lemma}
\begin{prop}
Let $(S, +, R)$ be a Rota--Baxter  Clifford semigroup. Then $+$ is the same as $\circ_R$ if and only if $a^{0}+R(a)\in C(S)$ for all $a\in S$.
\end{prop}
\begin{proof}Let $a,b\in S$. If  $+$ is the same as $\circ_R$,  then  $a\circ_R b=a+b$. Since
\begin{equation*}
\begin{aligned}
&a\circ b=a+b\\
&\Rightarrow a+R(a)+b-R(a)=a+R(a)-R(a)+b\\
&\Rightarrow a+R(a)+b-R(a)+R(a)=a+R(a)-R(a)+b+R(a)\\
&\Rightarrow a+R(a)+b=a+b+R(a)\\
&\Rightarrow -a+a+R(a)+b=-a+a+b+R(a)=b-a+a+R(a)\\
&\Rightarrow a^{0}+R(a)+b=b+a^{0}+R(a),\\
\end{aligned}
\end{equation*}
it follows that $a^{0}+R(a)\in C(S)$.
Conversely, if $a^{0}+R(a)\in C(S)$, then
\begin{equation*}
\begin{aligned}
a\circ b&=a+R(a)+b-R(a)=a+a^{0}+R(a)+b-R(a)\\
&=a+b+a^{0}+R(a)-R(a)=a+a^{0}+R(a)-R(a)+b=a+b.
\end{aligned}
\end{equation*}
This shows that $+$ is the same as $\circ_R$.
\end{proof}

A map $\psi$ from a dual weak left  brace $(T,+_{T},\circ_{T})$ to a dual weak  left brace$(S,+_{S},\circ_{S})$ is called  {\em   brace homomorphism} if
$$\psi(a+_{T} b)=\psi(a)+_{S}\psi(b) \mbox{ and } \psi(a\circ_{T} b)=\psi(a)\circ_{S}\psi(b) \mbox{ for all } a,b\in T.$$

\begin{lemma}\label{k} Let $(T,+_{T},\circ_{T})$ and $(S,+_{S},\circ_{S})$ be two dual weak left  braces and $\psi: (T, +_T)\rightarrow (S, +_S)$ be a semigroup homomorphism. Then $\psi$ is a  brace homomorphism if and only if  $\psi\lambda_{a}^{T}=\lambda_{\psi(a)}^{S}\psi$ for all $a\in T$.
\end{lemma}
\begin{proof}
By using Lemma \ref{wang5}, we can obtain
\begin{equation*}
\begin{aligned}
&\psi(-a+_{T}a\circ_{T}b)=\psi\lambda_{a}^{T}(b)=\lambda_{\psi(a)}^{S}\psi(b)
=-\psi(a)+_{S}(\psi(a)\circ_{S}\psi(b))\\
&\Longleftrightarrow \psi(-a+_{T}a\circ_{T}b)=-\psi(a)+_{S}(\psi(a)\circ_{S}\psi(b))\\
&\Longleftrightarrow \psi(a)+_{S}\psi(-a+_{T}a\circ_{T}b)=\psi(a)+_{S}(-\psi(a))+_{S}(\psi(a)\circ_{S}\psi(b))\\
&\Longleftrightarrow \psi(a^{0}+_{T}a\circ_{T}b)=\psi(a)^{0}+_{S}(\psi(a)\circ_{S}\psi(b))\\
&\Longleftrightarrow \psi(a^{0}\circ_{T}a\circ_{T}b)=\psi(a)^{0}\circ_{S}(\psi(a)\circ_{S}\psi(b))\\
&\Longleftrightarrow \psi(a\circ_{T}b)=\psi(a)\circ_{S}\psi(b).
\end{aligned}
\end{equation*}
for all $a,b\in T$, and so the result follows.
\end{proof}

To consider the quotient structures of dual weak  left braces, we need to recall  normal subsemigroups of inverse semigroups and give the congruences determined by them. Let $(S, +)$ be an inverse semigroup and $N$ be a nonempty set of $S$.  Then $N$ is called an {\em inverse subsemigroup} of $S$ if
$-m, m+n\in N$ for all $m, n\in N$.  Moreover, an inverse subsemigroup $N$ of $(S,+)$ is called {\em normal} if $E(S, +)\subseteq N$ and  $-a+n+a\in N$ for all $a\in S$ and $n\in N$.
By \cite[Theorem 5.3.3]{A10},  we have the following lemma.
\begin{lemma}\label{quotient}
 Let $(S, +)$ be an inverse semigroup and $N$   a normal subsemigroup of $S$ such that $-a+a=a-a$ for all $a\in N$.
Define a binary relation $\rho_N$ as follows: For all $a, b\in S$,
$$a~\rho_N~ b \mbox{ if and only if } a-a=b-b \mbox{ and } -a+b\in N.$$
Then $\rho_N$ is a congruence on $S$. In this case, the quotient inverse semigroup  $S/\rho_N$ is denoted simply by $S/N$.
Moreover, we denote the $\rho_N$-class containing $a\in S$ by $a+N$. Then $S/N=\{a+N\mid a\in S\}$, and
\begin{center}
$a+N=b+N$  if and only if $ a-a=b-b$ and $-a+b\in N$.
\end{center}
In particular,
if $(S, +)$ is a Clifford semigroup, then $-a+a=a-a=a^{0}$ for all $a\in N$ and
$$a~\rho_N~ b \mbox{ if and only if } a^{0}=b^{0} \mbox{ and } -a+b\in N \mbox{ for all } a, b\in S.$$
In this case, $(S/N, +)$ also forms a Clifford semigroup with the rule that $$(a+N)+(b+N)=(a+b)+N \mbox{ for all } a+N, b+N\in S/N.$$
\end{lemma}
\begin{defn}[\cite{c1}]
Let $(S,+, \circ)$ be a dual weak left brace and  $I$ be a non-empty subset of $S$. Then $I$ is called an {\em ideal} if the following conditions holds:
\begin{itemize}
\item[(1)] $I$ is a normal subsemigroup of $(S,+)$.
\item[(2)] $I$ is a normal subsemigroup of  $(S,\circ)$.
\item[(3)] $\lambda^S_{a}(I)\subseteq I$ for all $a\in S$.
\end{itemize}
\end{defn}
\begin{lemma}[\cite{c1}]Let $(S,+,\cdot)$ be a dual weak left brace and  $I$  an ideal of $(S,+,\cdot)$. Then $(S/I, +, \circ)$ forms a dual weak left brace with respect to the following operations:   $$(a+I)+ (b+I)=(a+b)+I,\,\, (a+I)\circ (b+I)=(a\circ b)+I$$   $\mbox{ for all } a+I, b+I\in S/I.$ In this case, $a+I=a\circ I$ for all $a\in S$.
\end{lemma}
Let $S$ be a non-empty set and $r_S: S\times S\longrightarrow S\times S$ be a map. Then for any $a\in S$, $r$ induces the following maps:
$$\lambda_{a}: S\rightarrow S, \,b\mapsto \mbox{the first component of }r(a,b),$$$$
\rho_{b}: S\rightarrow S,\, a\mapsto \mbox{the second component of }r(a,b).$$
Thus, for all $a,b\in S$, we have $r(a,b)=(\lambda_{a}(b),\rho_{b}(a))$.
If $r$ satisfies the following equality $$(r\times {\rm id}_{S})({\rm id}_{S}\times r)(r\times {\rm id}_{S})=({\rm id}_{S}\times r)(r\times
{\rm id}_{S})({\rm id}_{S}\times r)$$ in the set of maps from $S\times S\times S$ to $S\times S\times S$, where  ${\rm id}_{S}$ is the identity map on $S$, then  $r_S$ is called a {\em set-theoretic  solution} of the Yang-Baxter equation, or briefly a {\em solution}.
\begin{lemma}[Theorem 2.3 in \cite{Mazzotta} and Theorem 11 in \cite{Catino-Mazzotta-Miccoli-Stefanelli}]\label{jie}Let $(T, +, \circ)$ be a dual weak left  brace.  Then the map $r: S\times S\rightarrow S\times S$ defined by
$$r(a,b)=(-a+(a\circ b), (-a+(a\circ b))^{-1}\circ a\circ b) \mbox{ for all } a,b\in S$$
is a solution of the Yang-Baxter equation.
\end{lemma}

\section{Rota--Baxter Clifford semigroups}
In this section, we shall give some properties of Rota--Baxter Clifford semigroups whence the corresponding  results on Rota--Baxter groups given in \cite{b1,g2} are generalized to Rota--Baxter Clifford semigroups.
\subsection{Constructions  of Rota--Baxter operators }
In this subsection, we give some construction methods of Rota--Baxter operators on Clifford semigroups.
\begin{prop} Let $(S,+, R)$ be a Rota--Baxter Clifford semigroup  and $\varphi\in \Aut (S,+)$. Then  $R^{(\varphi)}=\varphi^{-1}R\varphi$ is  a  Rota--Baxter operator on $(S,+)$. If $R$ is strong, so is $R^{(\varphi)}$.
\end{prop}
\begin{proof}Let $a,b\in S$. Then
\begin{equation*}
\begin{aligned}
&\varphi(R^{(\varphi)}(a)+R^{(\varphi)}(b))=\varphi(\varphi^{-1}(R(\varphi(a)))+\varphi^{-1}(R(\varphi(b)))) \\
&=R(\varphi(a))+R(\varphi(b))=R(\varphi(a)+R(\varphi(a))+\varphi(b)-R(\varphi(a)))\\
&=R\varphi(a+\varphi^{-1}(R(\varphi(a)))+b+\varphi^{-1}(-R(\varphi(a))))\\
&=\varphi\varphi^{-1}R\varphi(a+\varphi^{-1}(R(\varphi(a)))+b-\varphi^{-1}(R(\varphi(a))))\\
&=\varphi[R^{(\varphi)}(a+R^{(\varphi)}(a)+b-R^{(\varphi)}(a))].
\end{aligned}
\end{equation*}
Since $\varphi$ is bijective, we have $$R^{(\varphi)}(a)+R^{(\varphi)}(b)=R^{(\varphi)}(a+R^{(\varphi)}(a)+b-R^{(\varphi)}(a)).$$
On the other hand,
\begin{equation*}
\begin{aligned}
a+R^{(\varphi)}(a)^{0}&=a+R^{(\varphi)}(a)-R^{(\varphi)}(a)\\
&=a+\varphi^{-1}(R(\varphi(a)))-\varphi^{-1}(R(\varphi(a)))\\
&=a+\varphi^{-1}(R(\varphi(a)))+\varphi^{-1}(-R(\varphi(a)))\\
&=\varphi^{-1}(\varphi(a)+R(\varphi(a))-R(\varphi(a)))=\varphi^{-1}(\varphi(a))=a.
\end{aligned}
\end{equation*}
Thus $R^{(\varphi)}$ is  a  Rota--Baxter operator. The final statement follows from Lemma \ref{strong}.
\end{proof}

Let $(S,+)$ be a  Clifford semigroup and $R: S\rightarrow S$   a map. Define a new map $\widetilde{R}$ as  follows: $\widetilde{R}: S\rightarrow S, a\mapsto -a+R(-a).$
\begin{prop} Let $(S,+, R)$ be a Rota--Baxter Clifford semigroup.
Then $\widetilde{R}$ is a Rota--Baxter operator on $(S,+)$ and $(\widetilde{R})^{(\varphi)}=\widetilde{R^{(\varphi)}}$ for every $\varphi\in \Aut (S,+)$. If $R$ is strong, so is $\widetilde{R}$.
\end{prop}
\begin{proof}The first part follows from the statements before Example 2 in \cite{c1}. To show the second part, let $a\in S$.  Then
$$
(\widetilde{R})^{(\varphi)}(a)=\varphi^{-1}(\widetilde{R}(\varphi(a)))=\varphi^{-1}(-\varphi(a)+R(-\varphi(a)))
=\varphi^{-1}(\varphi(-a)+R(\varphi(-a)))$$$$=-a+\varphi^{-1}R\varphi(-a)
=-a+R^{(\varphi)}(-a)=\widetilde{R^{(\varphi)}}(a).$$
This gives the desired result. The final statement follows from Lemma \ref{strong}.
\end{proof}

\begin{lemma}\label{b} Let $(S,+, R)$ be a Rota--Baxter Clifford semigroup and $a,b\in S$. If $R(a)\in E(S,+)$, then
$R(a+b)=R(a)+R(b).$
\end{lemma}
\begin{proof}Let $a,b\in S$. Then
$$R(a+b)=R(a+R(a)-R(a)+b)=R(a+R(a)+b-R(a))=R(a)+R(b).$$
Thus the results follows.
\end{proof}

\begin{lemma}
Let $(S,+, R)$ be a Rota--Baxter Clifford semigroup with $R\in \Aut (S,+)$ and $a\in S$. Then $R(a^{0}+R(a))\in C(S,+).$
\end{lemma}
\begin{proof}Let $x\in S$.  Since $R\in \Aut (S,+)$, we can let $x=R(b)$ for some $b\in S$. Observe that $R$ is a Rota--Baxter operator  and is an automorphism on $S$,  it follows that
$$R(a)+R(b)=R(a+R(a)+b-R(a))=R(a)+R(R(a))+R(b)-R(R(a)).$$
Adding  $R(R(a))$ on both sides in the above identity from the right, we have
$$R(a)+R(b)+R(R(a))=R(a)+R(R(a))+R(b).$$
Further adding $-R(a)$ on both sides  from the left, we get
$$-R(a)+R(a)+R(b)+R(R(a))=-R(a)+R(a)+R(R(a))+R(b).$$
Since $R$ is a homomorphism and $x=R(b)$,  and $R(a)^{0}=R(a^{0})$ by Lemma \ref{cwang} (1), it follows that
$$x+R(a^{0}+R(a))=x+R(a^{0})+R(R(a))=x+R(a)^{0}+R(R(a))$$$$=R(a)^{0}+R(R(a))+x=R(a^{0})+R(R(a))+x=R(a^{0}+R(a))+x.$$
This implies that $R(a^{0}+R(a))=R(a)^{0}+R(R(a))\in C(S)$.
\end{proof}

\begin{prop}
Let $(S,+)$ be a Clifford semigroup and $n$ be a positive integer.  Then $$R:S\rightarrow S,a\mapsto na=\underset{n}{\underbrace{a+a+\cdots +a}}$$
is a Rota-Baxter operator if and only if $n(a+b)=nb+na$ for all $ a,b\in S$. In this case, $R$ is strong.
\end{prop}
\begin{proof}
Let $a,b\in S$. Then $$ a+R(a)^{0}=a+R(a)-R(a)=a+na-na=a+a^{0}=a,$$
\begin{equation*}
\begin{aligned}
&R(a+R(a)+b-R(a))=n(a+na+b-na)\\
&=(n+1)a+b-na+(n+1)a+b-na+\cdots +(n+1)a+b-na\\
&=na+a+b+a+b+\cdots +a+b-na=na+n(a+b)-na.
\end{aligned}
\end{equation*}
On the other hand,
$$ R(a)+R(b)=R(a+R(a)+b-R(a))$$
\begin{equation*}
\begin{aligned}
&\Longleftrightarrow na+nb=na+n(a+b)-na\\
&\Longleftrightarrow -na+na+nb=n(a+b)-na\\
&\Longleftrightarrow -na+na+nb+na=n(a+b)-na+na\\
&\Longleftrightarrow nb+na=n(a+b).
\end{aligned}
\end{equation*}
Thus  the first part holds. Since $R(e)=ne=e$ for all $e\in E(S,+)$, it follows that $R$ is strong by Lemma \ref{strong}.
\end{proof}

\begin{prop}Let $(S,+)$ be a Clifford semigroup and  $b \in S$.  Then $R: S\rightarrow S,\, a\mapsto-b-a+b$ is a Rota-Baxter operator
if and only if $b^{0}$ is the identity in $S$ and $\{-b-x+b+x\mid x\in S\} \subseteq C(S).$  In this case, $R$ is strong.
\end{prop}
\begin{proof}Assume that $R$ is a Rota-Baxter operator and  $x,y\in S$. Then
$$x+R(x)^{0}=x-b-x+b-b+x+b$$$$=x-b-x+x+b=x-b+b=x+b^{0}=x.$$
 This shows that $b^{0}$ is the identity in $S$.  Moreover,
\begin{equation*}
\begin{aligned}
&R(x)+R(y)=R(x+R(x)+y-R(x))\\
&\Longleftrightarrow -b-x+b-b-y+b=-b-(x-b-x+b+y-b+x+b)+b\\
&\Longleftrightarrow -b-(y+x)+b=-b-b-x+b+x-x-y-b+x+b-x+b\\
&\Longleftrightarrow -x-y=(-b-x+b+x)-x-y+(-b+x+b-x)\\
&\Longleftrightarrow -(-b-x+b+x)-x-y=-x-y+(-b+x+b-x)\\
&\Longleftrightarrow -x+(-b+x+b-x)-y=-x-y+(-b+x+b-x)\\
&\Longleftrightarrow x-x+(-b+x+b-x)-y=x-x-y+(-b+x+b-x)\\
&\Longleftrightarrow (-b+x+b-x)-y=-y+(-b+x+b-x).\\
\end{aligned}
\end{equation*}
Since $x$ and $y$ are arbitrary, we have
\begin{equation} \label{wang10}
\{-b-x+b+x\mid x\in S\} \subseteq C(S).
\end{equation}
Conversely, suppose that $b^{0}$ is the identity in $S$ and (\ref{wang10}) holds.  Then
$$a+R(a)^{0}=a-b-a+b-b+a+b=a-b-a+a+b=a-b+b=a+b^{0}=a,$$
$$R(x)+R(y)=-b-x+b-b-y+b=-b-x-y+b,$$
\begin{equation*}
\begin{aligned}
&R(x+R(x)+y-R(x))\\
&=-b-(x-b-x+b+y-b+x+b)+b\\
&=-b-b-x+b-y-b+x+b-x+b\\
&=-b-b-x+b-b+x+b-x-y+b\\
&=-b-b-x+x+b-x-y+b\\
&=-b-b+b-x-y+b\\
&=-b-x-y+b=R(x)+R(y).
\end{aligned}
\end{equation*}
  Thus $R$ is a  Rota-Baxter operator. Since $$R(e)=-b-e+b=e-b+b=e+b^{0}=e$$ for all $e\in E(S,+)$, it follows that $R$ is strong by Lemma \ref{strong}.
\end{proof}
\begin{prop}Let $(S, +)$ be a Clifford semigroup and $a,b\in S$. Then $R: S\rightarrow S,\, x\mapsto a+x+b$ is a Rota-Baxter operator if and only if $S$ is a commutative monoid with identity $a^{0}$ and $b=-a$. In this case, $R$ is the identity map on $S$ and is strong by Lemma \ref{strong}.
\end{prop}
\begin{proof}Assume that $R$ is a Rota-Baxter operator. Then $a=a+R(a)^{0}$, and so
$$a^{0}=(a+R(a)^{0})^{0}=a^{0}+(R(a)^{0})^{0}=a^{0}+R(a)^{0}=a^{0}+(a+a+b)^{0}=a^{0}+a^{0}+a^{0}+b^{0}=a^{0}+b^{0}.$$
Similarly,  $b^{0}=b^{0}+a^{0}$. This implies that $a^{0}=a^{0}+b^{0}=b^{0}+a^{0}=b^{0}$. Let
$x\in S$. Then
$$x=x+R(x)^{0}=x+(a+x+b)^{0}=x+a^{0}+x^{0}+b^{0}=x+x^{0}+a^{0}+a^{0}=x+a^{0}=a^{0}+x.$$
So $a^{0}$ is the identity in $S$. Moreover,
$$a^{0}=(a+R(a)^{0})^{0}=a^{0}+(R(a)^{0})^{0}=a^{0}+R(a)^{0}=a^{0}+R(a^{0})=a^{0}+a+a^{0}+b
=a+b,$$
and so $-a=-a+a^{0}=-a+a+b=a^{0}+b=b^{0}+b=b.$ Thus $R(x)=a+x+b=a+x-a$ for all $x\in S$. In particular, $R(a)=a+a-a=a+a^{0}=a$. Let $x\in S$. Then
$$a+a+x-a=R(a)+R(x)=R(a+R(a)+x-R(a))=R(a+a+x-a)=a+a+a+x-a-a.$$
This implies that
$$
x=x+a^{0}=-a-a+(a+a+x-a)+a
$$$$=-a-a+(a+a+a+x-a-a)+a=a^{0}+a+x-a+a^{0}
=a+x-a,$$
and so $$x+a=a+x-a+a=a+x+a^{0}=a+a^{0}+x=a+x.$$ Therefore, $$R(x)=a+x-a=x+a-a=x+a^{0}=x$$ for all $x\in S$ as $a^{0}$ is the identity. Let
$x,y\in S$. Then
$$x+y=R(x)+R(y)=R(x+R(x)+y-R(x))=x+x+y-x.$$
This implies that
$$
y+x=y+x^{0}+x=x^{0}+y+x=-x+x+y+x
$$$$=-x+x+x+y-x+x=x^{0}+x+y+x^{0}=x+y+x^{0}
=x+y.
$$
Thus $S$ is commutative.

Conversely, assume the given conditions hold. Let $x,y\in S$. Then
$$R(x)=a+x+b=a+x-a=a-a+x=a^{0}+x=x,$$
$$
R(x)+R(y)=x+y=x+x^{0}+y=x+x-x+y=x+x+y-x$$$$=x+R(x)+y-R(x)=R(x+R(x)+y-R(x)).
$$
Moreover, $x+R(x)^{0}=x+x^{0}=x$. Thus $R$ is a Rota-Baxter operator.
\end{proof}
Now, we introduce Rota--Baxter operator with weight $-1$ on Clifford semigroups.
\begin{defn}Let $(S,+)$ be a Clifford semigroup and $L: S\rightarrow S$  be a map. Then $L$ is called a {\em Rota--Baxter operator with weight $-1$} on $(S,+)$ if
$$L(a)+L(b)=L(L(a)+b-L(a)+a),\,\,\, a+L(a)^{0}=a  \mbox{ for all } a,b\in S.$$
In addition,  a Rota--Baxter operator with weight $-1$ on $(S,+)$ is called {\em strong} if  $L(e)=e$ for all $e\in E(S,+)$.
\end{defn}
\begin{prop}Let $L$ be a Rota--Baxter operator with weight $-1$ on a Clifford semigroup $(S,+)$. Then $\widetilde{L}: S\rightarrow S,\, a\mapsto a+L(-a)$ is also a Rota--Baxter operator with weight $-1$. If $L$ is strong, so is $\widetilde{L}$.
\end{prop}
\begin{proof}Let $a,b\in S$. Then
\begin{equation*}
\begin{aligned}
&\widetilde{L}(a)+\widetilde{L}(b)=a+L(-a)+b+L(-b)\\
&=a+L(-a)-L(-a)+L(-a)+b+L(-b) \quad
 (\mbox{since }{L(-a)=L(-a)-L(-a)+L(-a)})\\
&=a+L(-a)+b-L(-a)+L(-a)+L(-b)  \quad  (\mbox{since } -L(-a)+L(-a)\in E(S))\\
&=a+L(-a)+b-L(-a)+L(L(-a)-b-L(-a)-a)\\
& (\mbox{since } L \mbox{ is a Rota--Baxter operator with weight}-1)\\
&=a+L(-a)+b-L(-a)+L(-(a+L(-a)+b-L(-a))) \\
&=\widetilde{L}(a+L(-a)+b-L(-a)) \quad (\mbox{by the definition of }\text{$\widetilde{L}$})\\
&=\widetilde{L}(a+L(-a)+b-L(-a)-a+a)  \quad\quad\quad (\mbox{since }a=a-a+a,-a+a\in E(S))\\
&=\widetilde{L}(\widetilde{L}(a)+b-\widetilde{L}(a)+a).  \quad (\mbox{by the definition of }\text{$\widetilde{L}$})\\
\end{aligned}
\end{equation*}
On the other hand,
$$a+\widetilde{L}(a)^{0}=a+\widetilde{L}(a)-\widetilde{L}(a)=a+a+L(-a)-L(-a)-a$$$$=a+a-a+L(-a)-L(-a)
=a+a-a=a-a+a=a.$$
Thus  $\widetilde{L}$ is  a Rota--Baxter operator with weight $-1$ on $(S,+)$.  If $L$ is strong, then we have $\widetilde{L}(e)=e+L(-e)=e-e=e$ for all $e\in E(S,+)$, and so $\widetilde{L}$ is also strong.
\end{proof}
\begin{prop}Let $(S, +, R)$ be a Rota--Baxter Clifford semigroup. Then $L: S\rightarrow S, \, a\mapsto R(-a)$ is a Rota--Baxter operator with weight $-1$. If $R$ is strong, so is $L$.
\end{prop}
\begin{proof}Let $a,b\in S$.  Then  we have
\begin{equation*}
\begin{aligned}
L(a)+L(b)&=R(-a)+R(-b)=R(-a+R(-a)-b-R(-a))\\
&=R(-(R(-a)+b-R(-a)+a))\\
&=L(R(-a)+b-R(-a)+a)=L(L(a)+b-L(a)+a),
\end{aligned}
\end{equation*}
$$
a+L(a)^{0}=a+L(a)-L(a)=a+R(-a)-R(-a) =a-a+a+R(-a)-R(-a)$$$$
 =a-a+R(-a)-R(-a)+a=a-a+a=a.
$$
Thus  ${L}$ is  a Rota--Baxter operator with weight $-1$ on $(S,+)$.  If $R$ is strong, then we have ${L}(e)=R(-e)=R(e)=e$ for all $e\in E(S,+)$, and so $L$ is also strong.
\end{proof}

\begin{lemma}\label{a}Let $(S, +, R)$ be a Rota--Baxter Clifford semigroup.  Then $L: S\rightarrow S, \, a\mapsto a+ R(a)$ is a Rota--Baxter operator with weight $-1$.  If $R$ is strong, so is $L$.
\end{lemma}
\begin{proof}Let $a,b\in S$. Then we have
\begin{equation*}
\begin{aligned}
&L(a)+L(b)=a+R(a)+b+R(b)=a+R(a)+b-R(a)+R(a)+R(b)\\
&=a+R(a)+b-R(a)+R(a+R(a)+b-R(a))=L(a+R(a)+b-R(a))\\
&=L(a+R(a)+b-R(a)-a+a)=L(L(a)+b-L(a)+a),
\end{aligned}
\end{equation*}
\begin{equation*}
\begin{aligned}
a+L(a)^{0}&=a+L(a)-L(a)=a+a+R(a)-R(a)-a=a+a-a=a.
\end{aligned}
\end{equation*}
Thus  ${L}$ is  a Rota--Baxter operator with weight $-1$ on $(S,+)$.  If $R$ is strong, then we have ${L}(e)=e+R(e)=e+e=e$ for all $e\in E(S,+)$, and so $L$ is also strong.
\end{proof}
Dually, we can prove the following lemma.
\begin{lemma}\label{adual} Let $L$ be a (strong) Rota--Baxter operator with weight $-1$ on a Clifford semigroup $(S,+)$. Then $R: S\rightarrow S, \, a\mapsto -a+ L(a)$ is a (strong) Rota--Baxter operator.
\end{lemma}

\begin{theorem}
Let $(S, +)$ be a Clifford semigroup. Denote
$$\mathcal{R}=\{R\mid R \mbox{ is a strong Rota--Baxter operator on } S\},$$
$$\mathcal{L}=\{L\mid L \mbox{ is a strong Rota--Baxter operator on } S \mbox{ with weight}-1\}.$$
Then there is a one-one correspondence between $\mathcal R$ and $\mathcal L$.
\end{theorem}
\begin{proof}By Lemmas \ref{a} and \ref{adual}, we can define $\Phi: \mathcal{R}\rightarrow \mathcal{L}$ and $\Psi: \mathcal{L}\rightarrow \mathcal{R}$, where
$\Phi(R): S\rightarrow S, a\mapsto a+R(a)$ and $\Psi(L): S\rightarrow S,a\mapsto -a+L(a)$ for all $R\in {\mathcal R}$ and $L\in {\mathcal L}$.  Let $R\in \mathcal{R}$ and $L\in\mathcal{L}$. Then by Lemma \ref{cwang} and the fact that $R$ is strong, we have
$$
(\Psi \Phi(R))(a) =-a+\Phi(R)(a)=-a+a+R(a)=a^{0}+R(a)
$$$$ =R(a^{0})+R(a)=R(a)^{0}+R(a)
 =R(a).$$
So $(\Psi\Phi)(R)=R$. Dually, we have $(\Phi\Psi)(L)=L$ for all $L\in \mathcal{L}$.
\end{proof}
\subsection{Substructures and quotient structures}

In this subsection, we consider some substructures and quotient structures of Rota--Baxter Clifford semigroups. We first give some notions and notations.
Let $(S, +)$ and $(S',+')$ be two semigroups and $\phi: S\mapsto S'$ be a map.  Then  $$\Ker \phi=\{x\in S\mid \phi(x)\in E(S',+')\} \mbox{ and }
\Ima \phi=\{\phi(a)\mid a\in S\}$$ are called the {\em kernel} and   {\em image} of $\phi$, respectively.

\begin{lemma}\label{c}Let $(S, +, R)$  be a Rota--Baxter  Clifford semigroup. Then $\Ker R$ and $\Ima R$ are Clifford subsemigroups of $(S,+)$.
\end{lemma}
\begin{proof}  Let  $a,b\in \Ker R$. Then $R(a), R(b)\in E(S,+)$. By Lemma \ref {b},  we have
$$R(a+b)=R(a)+R(b),\,\,\,  R(a+a)=R(a)+R(a)=R(a).$$ This implies that
\begin{equation*}
\begin{aligned}
&R(-a)=R(a-a-a)=R(a)+R(-a-a)=R(a+a)+R(-a-a)  \quad(\mbox{by Lemma \ref{b}})\\
&=R(a+a+R(a+a)-a-a-R(a+a))=R(a+a+R(a)-a-a-R(a)) \\
&=R(a+a-a-a+R(a)) \quad\quad(\mbox{since }R(a)\in E(S,+))\\
&=R(a^{0}+R(a)).\\
\end{aligned}
\end{equation*}
Since $E(S,+)$ forms a subsemigroup of $(S,+)$,  and $ R(E(S,+))\subseteq E(S,+)$ by Lemma \ref{cwang} (1),
we have $R(a+b)=R(a)+R(b)\in E(S, +)$ and $R(-a)=R(a^{0}+R(a))\in E(S)$ as
 $R(a),R(b),a^{0}\in E(S, +)$.  This implies that $a+b,-a\in \Ker R$, and so $\Ker R$ is a Clifford subsemigroup of $S$.

Next, we consider $\Ima R$.  Let $a,b\in \Ima R$ and $a=R(x)$ and $b=R(y)$ for some $x, y\in S$. Then by Lemma \ref{cwang} (4), we have
$$-a=-R(x)=R(-R(x)-x+R(x))\in \Ima R,$$
$$a+b=R(x)+R(y)=R(x+R(x)+y-R(x))\in \Ima R.$$
This shows that $\Ima R$ is also a Clifford subsemigroup of $S$.
\end{proof}

\begin{lemma}\label{d} Let $(S, +, R)$  be a Rota--Baxter   Clifford semigroup. Then  $$R_{+}: (S,\circ_R)\rightarrow (S,+),\, \, a\mapsto a+R(a)$$ is a semigroup homomorphism and $R_{+}R=RR_{+}$.
\end{lemma}
\begin{proof}Let $a,b\in S$. Then
\begin{equation*}
\begin{aligned}
R_{+}(a)+R_{+}(b)&=a+R(a)+b+R(b)=a+R(a)+b-R(a)+R(a)+R(b)\\
&=a\circ b+R(a+R(a)+b-R(a))=(a\circ b)+R(a\circ b)=R_{+}(a\circ b),
\end{aligned}
\end{equation*}
$$R_{+}(R(a))=R(a)+R(R(a))=R(a+R(a)+R(a)-R(a))=R(a+R(a))=R(R_{+}(a)).$$
This gives the desired result.
\end{proof}

\begin{theorem}Let $(S, +, R)$ be a Rota--Baxter  Clifford semigroup and $(S, +, \circ_R)$ be the dual weak left brace induced by $R$ (see Lemma \ref{wang9}). Then $$\Ker R=\{x\in S\mid R(x)\in E(S)\} \mbox{ and } \Ima R=\{R(a)\mid a\in S\},$$
 $$\Ker R_+=\{x\in S\mid R_+(x)\in E(S, +)\}=\{x\in S \mid R_+(x)\in E(S)\}.$$

Moreover, we have the following results:

\begin{itemize}
\item[(1)] $(S,+)=(\Ima R_{+},+)+ (\Ima R,+)$.
\item[(2)]$(\Ker R, \circ_R)$ and $(\Ker R_{+}, \circ_R)$ are normal subsemigroups of $(S,\circ_R)$.
\item[(3)] $(\Ker R, +)$ is a normal subsemigroup of $(\Ima R_{+}, +)$.
\item[(4)] If $R$ is strong, then $(\Ker R_+, +)$ is a  normal subsemigroup of $(\Ima R,+)$.
\item[(5)] If $R$ is strong, $(\Ima R_{+},+)/(\Ker R,+)\cong (\Ima R,+)/(\Ker R_{+},+)$.
\end{itemize}
\end{theorem}
\begin{proof}
(1)  Since $R$ is a Rota-Baxter operator on $(S,+)$, $\Ima R$ is a Clifford subsemigroup of $(S,+)$ by Lemma \ref{c}. On the other hand, let $a=R_{+}(x),b=R_{+}(y)\in \Ima R_{+}$ where $x,y\in S$. Then
$-a=-R_{+}(x)=R_{+}(x^{-1})\in \Ima R_{+}$ and
$$a+b=R_{+}(x)+R_{+}(y)=R_{+}(x\circ_{R}y)\in \Ima R_{+}$$ by Lemma \ref{d}. This implies that $\Ima R_{+}$ is a Clifford subsemigroup of $(S,+)$.
Obviously, we have $\Ima R_{+}+\Ima R\subseteq S$. Moreover, let $a\in S$. Then
$$a=a+R(a)^{0}=a+R(a)-R(a)=R_{+}(a)+(-R(a))\in \Ima R_{+}+\Ima R.$$
So $S\subseteq \Ima R_{+}+\Ima R$. This implies that $S=\Ima R_{+}+\Ima R$.

(2) By Lemma \ref{wang9}, $R$ is a  Rota--Baxter operator on $(S, \circ_R)$. In view of Lemma \ref{cwang} (1), we have $R(E(S))\subseteq E(S)$, and so $E(S)\subseteq \Ker R$. By Lemmas \ref{wang9} and \ref{c}, $\Ker R$ is a Clifford subsemigroup of $(S, \circ_R)$. Let $a\in S$ and $n\in \Ker R$.
Then $R(n)\in E(S)$ and $$R(a^{-1}\circ_R n\circ_R a)=R(a^{-1})+R(n)+R(a)=R(a^{-1})+R(a)+R(n)$$$$=R(a^{-1}\circ_R a)+R(n)  \subseteq E(S)+R(n)\subseteq E(S)$$ by Lemma \ref{wang9} (3), which gives that $a^{-1}\circ_R n\circ_R a\in \Ker R$. Thus $(\Ker R, \circ_R)$ is a normal subsemigroup  of $(S,\circ_R)$.
On the other hand,
Since $R_+$ is a semigroup homomorphism between the Clifford semigroups $(S, \circ_R)$ and  $(S, +)$, we can show that $\Ker R_+$ is also a   normal subsemigroup  of $(S,\circ_R)$ by similar arguments as above.

(3)  If $a\in \Ker R$, then $R(a)\in E(S,+)$. This implies that $$a=a+R(a)^{0}=a+R(a)=R_{+}(a)\in \Ima R_{+}.$$
So $\Ker R\subseteq \Ima R_{+}$.  By item (1) in the present Theorem, we have $E(S)\subseteq \Ker R$. Since  $R$ is a Rota-Baxter operator on $(S,+)$, $(\Ker R,+)$ is a Clifford subsemigroup of $(S,+)$ by Lemma \ref{c}.
Now let $a\in \Ker R$ and $b\in S$. Then $R(a)\in E(S)$ and
\begin{equation*}
\begin{aligned}
&b\circ_R a\circ_R b^{-1}=(b+R(b)+a-R(b))\circ_R(-R(b)-b+R(b))\\
&=b+R(b)+a-R(b)+R(b\circ_R a)-R(b)-b+R(b)-R(b\circ_R a)\\
&=b+R(b)+a-R(b)+R(b)+R(a)-R(b)-b+R(b)-R(a)-R(b)\\
&=b+R(b)+a-R(b)-b+R(b)-R(b)=R_{+}(b)+a-R_{+}(b).
\end{aligned}
\end{equation*}
This implies that
$$R(R_{+}(b)+a-R_{+}(b)) =R(b\circ_R a\circ_R b^{-1})=R(b)+R(a)+R(b^{-1})
 =R(b)^{0}+R(a)\in E(S),$$
which gives  that $ R_{+}(b)+a-R_{+}(b)\in \Ker R$. Thus $(\Ker R,+)$ is a normal subsemigroup of $(\Ima R_{+},+)$.

(4) Let $a\in \Ker R_{+}$. Then $a+R(a)=R_{+}({a})\in E(S)=E(S,+)$. Since $R$ is a strong Rota-Baxter operator on $(S,+)$, we have $R(a)^{0}=R(a^{0})=a^{0}$ by Lemmas \ref{cwang} and \ref{strong}. As $a+R(a)\in E(S,+)$, we have
$$a+R(a)=(a+R(a))^{0}=a^{0}+R(a)^{0}=a^{0}+a^{0}=a^{0},$$
which implies
$$a=a+R(a)^{0}=a+R(a)-R(a)=a^{0}-R(a)=R(a)^{0}-R(a)=-R(a).$$
Moreover,
\begin{equation*}
\begin{aligned}
-R(a)&=R(-R(a)-a+R(a))=R(-(a+R(a))+R(a))=R(-a^{0}+R(a))\\
&=R(a^{0}+R(a))=R(R(a)^{0}+R(a))=R(R(a))=R(-a).
\end{aligned}
\end{equation*}
This yields that $a=R(-a)\in \Ima R$. By (1) in the present lemma, we have $E(S,\circ_{R})=E(S,+)\subseteq \Ker R_{+}$.

Let $a,b\in \Ker R_{+}$. Then $a+R(a),b+R(b)\in E(S,+)$, and so $a=R(-a)=-R(a)$ and $b=R(-b)=-R(b)$ by previous paragraph. This implies that
$$R_{+}(-a)=-a+R(-a)=-a+a=a^{0}\in E(S,+),$$ whence $-a\in \Ker R_{+}$.
Moreover,
\begin{equation*}
\begin{aligned}
&R_{+}(a+b)=a+b+R(a+b)=-R(a)-R(b)+R(a+b)\\
&=-(R(b)+R(a))+R(a+b)
=-R(b+R(b)+a-R(b))+R(a+b)\\
&=-R(b-b+a+b)+R(a+b)
 =-R(a+b)+R(a+b)\in E(S,+),
\end{aligned}
\end{equation*}
and so $a+b\in \Ker R_{+}$. Let $a\in \Ker R_{+}$ and $R(b)\in \Ima R, b\in S$.    Then
$a=R(-a)$ and
\begin{equation*}
\begin{aligned}
&R_{+}(R(b)+a-R(b))=R(b)+a-R(b)+R(R(b)+a-R(b))\\
&=R(b)+R(-a)+R(-R(b)-b+R(b)+a)\\
&=R(b)+R(-a+a-R(b)-b+R(b)+a-a)\\
&=R(b)+R(a^{0}-R(b)-b+R(b))\\
&=R(b+R(b)+a^{0}-R(b)-b+R(b)-R(b))\\
&=R(a^{0}+b^{0}+R(b)^{0}) =R(a^{0}+b^{0})=a^{0}+b^{0}\in E(S,+).
\end{aligned}
\end{equation*}
This shows  $R(b)+a-R(b)\in \Ker R_{+}$. Thus $(\Ker R_{+},+)$ is a normal subsemigroup of $(S,+)$.

(5) By items (2) and (3), the two quotient semigroups in question are meaningful. Define
$$\varphi: \Ima R_{+}/\Ker R\rightarrow \Ima R/\Ker R_{+},\quad R_{+}(a)+\Ker R\mapsto R(a)+\Ker R_{+}.$$
Let $a,b\in S$. Then
\begin{equation*}
\begin{aligned}
R_{+}(a)^{0}=R_{+}(b)^{0}&\Longleftrightarrow (a+R(a))^{0}=(b+R(b))^{0}
\Longleftrightarrow a^{0}+R(a)^{0}=b^{0}+R(b)^{0} \\
&\Longleftrightarrow a^{0}=b^{0}\Longleftrightarrow R(a^{0})=R(b^{0}) \Longleftrightarrow R(a)^{0}=R(b)^{0}.
\end{aligned}
\end{equation*}
On the other hard, we have $R(a^{-1})=-R(a)$, $R_{+}(a^{-1})=-R_{+}(a)$  and
$$
R(-R_{+}(a)+R_{+}(b))=R(R_{+}(a^{-1})+R_{+}(b))=R(R_{+}(a^{-1}\circ_{R}b))
$$$$=RR_{+}(a^{-1}\circ_{R}b)=R_{+}R(a^{-1}\circ_{R}b)
=R_{+}(R(a^{-1})+R(b))=R_{+}(-R(a)+R(b))
$$
 by Lemmas \ref{wang9} and \ref{d}.
This implies that $-R_{+}(a)+R_{+}(b)\in \Ker R$ if and only if $ -R(a)+R(b)\in \Ker R_{+}$.
Thus
$$ R_{+}(a)+\Ker R=R_{+}(b)+\Ker R\\
\Longleftrightarrow R_{+}(a)^{0}=R_{+}(b)^{0}, -R_{+}(a)+R_{+}(b)\in \Ker R
$$$$\Longleftrightarrow R(a)^{0}=R(b)^{0}, -R(a)+R(b)\in \Ker R_{+}
\Longleftrightarrow R(a)+\Ker R_{+}=R(b)+\Ker R_{+}.$$
This shows that $\varphi$ is an injection. Obviously, $\varphi$ is surjective.
Finally,
$$
\varphi((R_{+}(a)+\Ker R)+(R_{+}(b)+\Ker R))=\varphi((R_{+}(a)+R_{+}(b))+\Ker R)
$$$$=\varphi(R_{+}(a\circ_{R}b)+\Ker R)=R(a\circ_{R}b)+\Ker R_{+}
=(R(a)+R(b))+\Ker R_{+}
$$$$=(R(a)+\Ker R_{+})+(R(b)+\Ker R_{+})
=\varphi(R_{+}(a)+\Ker R)+\varphi(R_{+}(b)+\Ker R).
$$
Thus $\varphi$ is a  semigroup isomorphism.
\end{proof}

Let $(S, +)$ be a Clifford semigroup and $U$ and $V$   two Clifford subsemigroups of $S$.  Then $(U,V)$ is called an {\em exact factorization} of $S$ if any element $a$ in $S$ can be written in a unique way as $a=u_a+v_a$ with $u_a\in U$ and $v_a\in V$. In this case, define
$$R_{U,V}: S\rightarrow S,\,\, a=u_{a}+v_{a}\mapsto-v_{a}.$$ Then by \cite[Example 2]{c1}, $R_{U,V}$ is a Rota-Baxter operator on $(S,+)$.

\begin{prop}Let $(U, V)$ be an  exact factorization of a Clifford semigroup $(S,+)$ and $R=R_{U,V}$.
\begin{itemize}
\item[(1)] $R(a+R(a))\in E(S,+)$.
\item[(2)]$(S,\circ_R)\cong U\times V^{\rm op}$, where $V^{\rm op}$ is the dual semigroup of $V$.
\end{itemize}
\end{prop}
\begin{proof}Let $a=u_{a}+v_{a}$ and $b=u_{b}+v_{b}\in S$ with $u_a, u_b\in U$ and $v_a,v_b\in V$. Item (1) follows from the equation
$$R(a+R(a))=R(u_{a}+v_{a}-v_{a})=-(v_{a}-v_{a})=v_{a}-v_{a}=v_{a}^{0}\in E(S,+).$$
Since $(U, V)$ is an  exact factorization of $S$, the map
$$\varphi:(S,\circ_R)\rightarrow U\times V^{\rm op},\,\, a=u_{a}+v_{a}\mapsto (u_{a},v_{a})$$
is a bijection.   Moreover,
\begin{equation*}
\begin{aligned}
&\varphi(a\circ_R b)=\varphi(a+R(a)+b-R(a))=\varphi(u_{a}+v_{a}-v_{a}+u_{b}+v_{b}+v_{a})\\
&=\varphi(u_{a}+u_{b}+v_{b}+v_{a})=(u_{a}+u_{b},v_{b}+v_{a})=(u_{a},v_{a})(u_{b},v_{b})=\varphi(a)\varphi(b).
\end{aligned}
\end{equation*}
This shows that $\varphi$ is also a homomorphism.  Thus (2) also holds.
\end{proof}

\begin{prop}
Let $(S, +)$ be a Clifford semigroup, $U,V,T$ be Clifford subsemigroups of $S$ and $L$ is a Rota-Baxter operator on $V$  such that any element $a$ in $S$ can be written in a unique way as $a=u_a+ v_a+ t_a$ with $u_a\in U, v_a\in V, t_a\in T$ and $uv=vu$ and $L(v)t=tL(x)$ for all $u\in U, v\in V$ and $t\in T$. Then $$R:S\rightarrow S:u_{a}+v_{a}+t_{a}\mapsto L(v_{a})-t_{a}$$ is a Rota-Baxter operator on  $S$. If this is the case, we have $(S,\circ_R)\cong U\times V_{L}\times T^{\rm op}$, where $V_L=(V, \circ_L)$ and where $T^{\rm op}$ is the dual semigroup of $T$.
\end{prop}
\begin{proof} By hypothesis,  $R$ is well defined. Let $a=u_{a}+v_{a}+t_{a},\, b=u_{b}+v_{b}+t_{b}\in S$. Then
\begin{equation*}
\begin{aligned}
&R(a)+R(b)=R(u_{a}+v_{a}+t_{a})+R(u_{b}+v_{b}+t_{b})=L(v_{a})-t_{a}+L({v_{b}})-t_{b}\\
&=L(v_{a})+L(v_{b})-t_{a}-t_{b}=L(v_{a}+L(v_{a})+v_{b}-L(v_{a}))-(t_{b}+t_{a})\\
&=R(u_{a}+u_{b}+v_{a}+L(v_{a})+v_{b}-L(v_{a})+t_{b}+t_{a}),
\end{aligned}
\end{equation*}
\begin{equation*}
\begin{aligned}
&R(a+R(a)+b-R(a))=R(u_{a}+v_{a}+t_{a}+L(v_{a})-t_{a}+u_{b}+v_{b}+t_{b}+t_{a}-L(v_{a}))\\
&=R(u_{a}+v_{a}+L(v_{a})+t_{a}-t_{a}+u_{b}+v_{b}-L(v_{a})+t_{b}+t_{a})\\
&=R(u_{a}+v_{a}+L(v_{a})+u_{b}+v_{b}-L(v_{a})+t_{b}+t_{a})\\
&=R(u_{a}+u_{b}+v_{a}+L(v_{a})+v_{b}-L(v_{a})+t_{b}+t_{a})=R(a)+R(b),
\end{aligned}
\end{equation*}
\begin{equation*}
\begin{aligned}
a+R(a)^{0}&=a+R(a)-R(a)=u_{a}+v_{a}+t_{a}+L(v_{a})-t_{a}+t_{a}-L(v_{a})\\
&=u_{a}+v_{a}+t_{a}+L(v_{a})-L(v_{a})=u_{a}+v_{a}+t_{a}=a.
\end{aligned}
\end{equation*}
This shows that $R$ is a Rota-Baxter operator on  $S$. Define
$$\varphi:(S,\circ_R)\rightarrow U\times V_{L}\times T^{\rm op}, a=u_{a}+v_{a}+t_{a}\mapsto (u_{a},v_{a},t_{a}).$$
Then $\varphi$ is well-defined by hypothesis. Moreover,
\begin{equation*}
\begin{aligned}
&\varphi(a\circ_R b)=\varphi(a+R(a)+b-R(a))\\
&=\varphi(u_{a}+v_{a}+t_{a}+L(v_{a})-t_{a}+u_{b}+v_{b}+t_{b}+t_{a}-L(v_{a}))\\
&=\varphi(u_{a}+v_{a}+L(v_{a})+t_{a}-t_{a}+u_{b}+v_{b}-L(v_{a})+t_{b}+t_{a})\\
&=\varphi(u_{a}+v_{a}+L(v_{a})+u_{b}+v_{b}-L(v_{a})+t_{b}+t_{a})\\
&=\varphi(u_{a}+u_{b}+v_{a}+L(v_{a})+v_{b}-L(v_{a})+t_{b}+t_{a})\\
&=(u_{a}+u_{b},v_{a}+L(v_{a})+v_{b}-L(v_{a}),t_{b}+t_{a})\\
&=(u_{a}+u_{b},v_{a}\circ_{L}v_{b},t_{b}+t_{a})=(u_{a},v_{a},t_{a})(u_{b},v_{b},t_{b})=\varphi(a)\varphi(b).
\end{aligned}
\end{equation*}
This gives that $\varphi$ is an isomorphism.
\end{proof}

\begin{prop}\label{wang1}Let $R$ be a Rota--Baxter operator on a commutative Clifford semigroup $(S, +)$. Then $R$ is also an endomorphism on $(S,+)$.
\end{prop}
\begin{proof}Let $a,b\in S$. Then
$$R(a)+R(b)=R(a+R(a)+b-R(a))=R(a+R(a)-R(a)+b)=R(a+b).$$
This shows that $R$ is also an endomorphism on $(S,+)$.
\end{proof}
\begin{remark}
The converse of Proposition \ref{wang1} is not true. Let  $(S, +)$ be a commutative Clifford semigroup,  where $S=\{e,f\}$ and $e+f=f+e=e$ and $ff=f$. Define $R(x)=e$ for all $x\in S$. Then it is easy to see that $R$ is an endomorphism on $(S,+)$. However, $R$ is not a  Rota--Baxter operator as $f+R(f)^{0}=f+e=e\not=f$.
\end{remark}

\section{Post Clifford semigroups and dual weak left braces}
In this section, we introduce post Clifford semigroups and show that the category  of strong post Clifford semigroups  is isomorphic to the category of dual weak left braces. Let $(T,+)$ be an Clifford semigroup and  $a,b\in T$. Denote $H_{a}=\{x\in T\mid x^{0}=a^{0}\}$. Then it is easy to see that $(H_{a},+)$ is a subgroup of $(T,+)$ and $a^{0}$ is the identity in $(H_{a},+)$ by Lemma \ref{jichu}. Obviously, $H_{a}=H_{b}$ if and only if $a^{0}=b^{0}$.

\begin{defn}
Let $(T,+)$ be a Clifford semigroup and $\rhd$ be a binary operation on $T$. For each $a\in T$, define $L_{a}^{\rhd}: T\rightarrow T,b\mapsto a\rhd b$. Then $(T,+,\rhd)$ is called a {\em post Clifford  semigroup} if the following conditions holds:  For all $a,b,c\in T$,
\begin{itemize}
\item[(P1)] $L_{a}^{\rhd}\in {\rm End}(T,+)$, that is, $a\rhd (b+c)=(a\rhd b)+(a\rhd c)$.
\item[(P2)] $(a+(a\rhd b))\rhd c=a\rhd(b\rhd c)$.
\item[(P3)] $a+(a\rhd b^{0})=b^{0}+(b^{0}\rhd a)$.
\item[(P4)] $L_{a}^{\rhd}(H_{a})\subseteq H_{a}$ and
$L_{a}^{\rhd}|_{H_{a}}: H_{a}\rightarrow H_{a},b\mapsto L_{a}^{\rhd}(b)=a\rhd b$
is an automorphism of $(H_{a},+)$.
\end{itemize}
A post Clifford semigroup $(T,+,\rhd)$ is called {\em strong} if
\begin{itemize}
\item[(P5)]$a^{0}+(a\rhd b)=a\rhd b$ for all $a,b\in T$.
\end{itemize}
\end{defn}
\begin{exam}Let $(T,+)$ be an idempotent commutative semigroup and let $a\rhd b= a+b$ for all $a,b\in T$.  Then $a^0=a$ for all $a\in S$. It is routine to check that
$(T, +, \rhd)$ forms a strong post Clifford  semigroup.
\end{exam}
\begin{exam}Let $(T,+)$ be a group with identity $e$ and $\rhd$ be a binary operation on $T$. Recall from \cite{Bai} that $(T, +, \rhd)$ is called a {\em post-group} if (P2) holds and
$L_{a}^{\rhd}: T\rightarrow T, b\mapsto a\rhd b$ is an automorphism of $(T,+)$ for all $a\in T$.
Observe that $(T,+)$ is a  Clifford semigroup, $a^0=e$ and $H_a=T$ for all $a\in T$ in this case. This implies that (P1) and (P4) hold. Moreover, for all $a,b\in T,$ by \cite[Lemma 2.2]{Bai} we have $a\rhd e=e$ and $e\rhd a=a$, and so $$a+(a\rhd  b^0)=a+(a\rhd e)=a+e=e+a=e+(e\rhd a)=b^0+(b^0\rhd a).$$
This gives (P3). Thus post-groups are all post Clifford  semigroups.
\end{exam}

\begin{lemma}\label{q2}
Let $(T,+,\rhd)$ be a post Clifford semigroup and $a\in T$. Then $$L_{a}^{\rhd}(a^{0})=a\rhd a^{0}=a^{0}=a^{0}\rhd a^{0}=L_{a^{0}}^{\rhd}(a^{0}),\, L_{a^{0}}^{\rhd}(a)=a^{0}\rhd a=a.$$
\end{lemma}
\begin{proof}
By (P4), $L_{a}^{\rhd}|_{H_{a}}\in {\rm Aut}(H_{a},+)$. Since $a,a^{0}\in H_{a}$ and $a^{0}$ is the identity in the group $(H_{a},+)$, we have
$a\rhd a^{0}=L_{a}^{\rhd}(a^{0})=L_{a}^{\rhd}|_{H_{a}}(a^{0})=a^{0}.$ Replacing $a$ by $a^{0}$ in the above equation, we have
$a^{0}\rhd a^{0}=a^{0}\rhd a^{00}=a^{00}=a^{0}.$ On the other hand,
\begin{equation} \label{q1}
L_{a^{0}}^{\rhd}(a)=a^{0}\rhd a=(a^{0}+a^{0})\rhd a=(a^{0}+(a^{0}\rhd a^{0}))\rhd a
=a^{0}\rhd (a^{0}\rhd a)=L_{a^{0}}^{\rhd}(a^{0}\rhd a)
\end{equation}
by (P2). Since $a,a^{0}\in {H}_a$, we get $a^{0}\rhd a=L_{a^{0}}^{\rhd}(a)\in H_{a}$ by (P4). Observe that $L_{a^{0}}^{\rhd}|_{H_{a}}\in {\rm Aut}(H_{a},+)$ by (P4) again, it follows that $a=a^{0}\rhd a$ by (\ref{q1}).
\end{proof}
A map $\psi$ from a post Clifford semigroup $(T,+_{T},\rhd_{T})$ to a  post Clifford semigroup $(S,+_{S},\rhd_{S})$ is called  {\em  post Clifford semigroup homomorphism} if
$$\psi(a+_{T} b)=\psi(a)+_{S}\psi(b) \mbox{ and } \psi(a\rhd_{T} b)=\psi(a)\rhd_{S}\psi(b) \mbox{ for all } a,b\in T.$$

\begin{theorem}\label{q3}
Let $(T,+,\rhd)$ be a post Clifford semigroup. Define a binary operation $``\circ"$ on $T$ as follows: For all $a,b\in T$,
$$a\circ b=a+(a\rhd b)=a+L_{a}^{\rhd}(b).$$
\begin{itemize}
\item[(1)] $(T,\circ)$ forms a Clifford semigroup, and for all $a\in T$, the inverse of $a$ in $(T,\circ)$ is $a^{-1}=(L_{a}^{\rhd}|_{H_{a}})^{-1}(-a)$ and $-a+a=a^{0}=a\circ a^{-1}$. Moreover, $E(T,+)=E(T,\circ)$.
\item[(2)] $L^{\rhd}: (T,\circ)\rightarrow {\rm End}(T,+), a\mapsto L_{a}^{\rhd}$ is a semigroup homomorphism.
\item[(3)] If $\psi:(T,+,\rhd)\rightarrow (S,+,\rhd)$ is a post Clifford semigroup homomorphism, then $\psi$ is a semigroup homomorphism  from $(T,\circ)$ to $(S,\circ)$.
\end{itemize}
\end{theorem}
\begin{proof}We first observe that $x^0=-x+x=x-x$ for all $x\in T$. Let $a,b,c\in T$.

(1)  By (P1) and (P2), we have
\begin{equation*}
\begin{aligned}
&(a\circ b)\circ c=(a\circ b)+((a\circ b)\rhd c)=a+(a\rhd b)+((a+(a\rhd b))\rhd c)\overset{\rm (P2)}{=}\\
&a+(a\rhd b)+(a\rhd(b\rhd c))\overset{\rm(P1)}{=}a+(a\rhd (b+(b\rhd c)))=a\circ (b+(b\rhd c))=a\circ (b\circ c).
\end{aligned}
\end{equation*}
Since $-a\in H_{a}$, we denote $a^{-1}=(L_{a}^{\rhd}|_{H_{a}})^{-1}(-a)$ by (P4). Then $a^{-1}\in H_{a}, (a^{-1})^{0}=a^{0}$ and
$$a\circ a^{-1}=a+(a\rhd a^{-1})=a+L_{a}^{\rhd}(a^{-1})=a+L_{a}^{\rhd}((L_{a}^{\rhd}|_{H_{a}})^{-1}(-a))=a-a=a^{0}.$$ This together with
Lemma \ref{q2} implies that
$$(a\circ a^{-1})\circ a=a^{0}\circ a=a^{0}+(a^{0}\rhd a)=a^{0}+a=a,$$
$$
a^{-1}\circ a\circ a^{-1}=a^{-1}\circ a^{0}=a^{-1}+(a^{-1}\rhd a^{0})
=a^{-1}+(a^{-1}\rhd (a^{-1})^{0})=a^{-1}+(a^{-1})^{0}
=a^{-1}.
$$
Let $a\in E(T,\circ)$. Then $a=a\circ a=a+(a\rhd a)=a+L_{a}^{\rhd}(a)$ and $L_{a}^{\rhd}(a)\in H_{a}$ by (P4), and so
$$a^{0}=-a+a=-a+a+L_{a}^{\rhd}(a)=a^{0}+L_{a}^{\rhd}(a)=L_{a}^{\rhd}(a)$$ by Lemma \ref{q2}. Since
$a\in H_{a}$ and $L_{a}^{\rhd}|_{H_{a}}\in {\rm Aut}(H_{a},+)$, we have $a=a^{0}\in E(T,+)$.  Conversely,  let $e\in E(T,+)$. Then $e=e^{0}$, and so
$$e\circ e=e+(e\rhd e)=e+(e\rhd e^{0})=e+e^{0}=e+e=e\in E(T,\circ)$$ by Lemma \ref{q2}. Thus $E(T,+)=E(T,\circ)$.
Now, let $e\in E(T,\circ)=E(T,+)$. Then $e=e^{0}$.  By (P3),
$$a\circ e=a+(a\rhd e)=a+(a\rhd e^{0})=e^{0}+(e^{0}\rhd a)=e^{0}\circ a=e\circ a.$$
This implies that $(T,\circ)$ is a Clifford semigroup and $a^{-1}=(L_{a}^{\rhd}|_{H_{a}})^{-1}(-a)$  for all $a\in T$.

(2) Observe that $$L_{a\circ b}^{\rhd}(c)=(a+(a\rhd b))\rhd c=a\rhd (b\rhd c)=L_{a}^{\rhd}(L_{b}^{\rhd}(c))=(L_{a}^{\rhd}L_{b}^{\rhd})(c),$$
it follows that $L^{\rhd}(a\circ b)=L_{a\circ b}^{\rhd}=L_{a}^{\rhd}L_{b}^{\rhd}=L^{\rhd}(a)L^{\rhd}(b)$. This gives that $L^{\rhd}$ is a semigroup homomorphism.
\item[(3)] Since $$\psi(a\circ b)=\psi(a+(a\rhd b))=\psi(a)+\psi(a\rhd b)=\psi(a)+(\psi(a)\rhd \psi(b))=\psi(a)\circ \psi(b),$$  the desired result holds.
\end{proof}
\begin{defn}
Let $(T,+,\rhd)$ be a post Clifford semigroup. Then the Clifford semigroup $(T,\circ)$ obtained in Theorem \ref{q3} is called the {\em  sub-adjacent Clifford semigroup} of $(T,+,\rhd)$.
\end{defn}
In the sequel, we shall consider the connection between  post Clifford semigroups and dual weak left braces. Recall that a map $\psi$ from a dual weak left  brace $(T,+_{T},\circ_{T})$ to a dual weak  left brace$(S,+_{S},\circ_{S})$ is called a {\em   dual weak  left brace homomorphism} if
$$\psi(a+_{T} b)=\psi(a)+_{S}\psi(b) \mbox{ and } \psi(a\circ_{T} b)=\psi(a)\circ_{S}\psi(b) \mbox{ for all } a,b\in T.$$
\begin{lemma}\label{q4}
Let $(T,+,\rhd)$ be a post Clifford semigroup with the sub-adjacent Clifford semigroup $(T,\circ)$. Then $(T,+,\circ)$ is a dual weak left brace. If $\psi: (T,+,\rhd) \rightarrow (S,+,\rhd)$ is a post Clifford semigroup homomorphism, then $\psi$ is a dual weak left brace homomorphism from $(T,+,\circ)$ to $(S,+,\circ)$.
\end{lemma}
\begin{proof}
Let $a,b,c\in T$. Then
$$
a\circ(b+c)=a+(a\rhd(b+c))\overset{\rm(P1)}=a+(a\rhd b)+(a\rhd c)
=a-a+a+(a\rhd b)+(a\rhd c)$$$$=a+(a\rhd b)-a+a+(a\rhd c)
=(a\circ b)-a+(a\circ c).$$
By Theorem \ref{q3} (1), we have $a\circ a^{-1}=-a+a$. Moreover, if $\psi: (T,+,\rhd) \rightarrow (S,+,\rhd)$ is a post Clifford semigroup homomorphism,
$$\psi(a\circ b)=\psi(a+(a\rhd b))=\psi(a)+\psi(a\rhd b)=\psi(a)+(\psi(a)\rhd \psi(b))=\psi(a)\circ \psi(b).$$
This gives that $\psi$ is a dual weak left brace homomorphism.
\end{proof}
\begin{theorem}\label{postjie}Let $(T,+,\rhd)$ be a post Clifford semigroup with sub-adjacent Clifford semigroup $(T,\circ)$. Then  the map $r: T\times T\rightarrow T\times T$ given  by
$$r(a,b)=(-a+a+(a\rhd  b),\,\, (a\rhd  b)^{-1}\circ  a\circ b)=(a^0+(a\rhd  b),\,\, (a\rhd  b)^{-1}\circ  a\circ b)
$$$$=(-a+a+(a\rhd  b),\,\, (L^\rhd_{a\rhd b}|_{H_{a\rhd b}})^{-1}(-(a\rhd b))+((L^\rhd_{a\rhd b}|_{H_{a\rhd b}})^{-1}(-(a\rhd b))\rhd (a+(a\rhd b))))$$
$\mbox{ for all } a,b\in T$, is a solution of the Yang-Baxter equation.
\end{theorem}
\begin{proof} Let $a,b\in T$.  Then $(T,+,\circ)$ is a dual weak left brace by Lemma \ref{q4}.
In view of Lemma  \ref{wang5}, we have
$-a+a\circ b=-a+a+(a\rhd  b)=a^0+(a\rhd  b)=a^0\circ (a\rhd  b),$ and so $$(-a+a\circ b)^{-1}\circ a\circ b=(a^0\circ (a\rhd  b))^{-1}\circ a\circ b=(a\rhd  b)^{-1} \circ(a^0)^{-1}\circ a\circ b$$$$=(a\rhd  b)^{-1} \circ  a^0 \circ a\circ b=(a\rhd  b)^{-1}\circ  a\circ b=(a\rhd  b)^{-1}\circ (a+(a\rhd b))$$$$=(a\rhd  b)^{-1}+ (a\rhd  b)^{-1}\rhd (a+(a\rhd b))=(L^\rhd_{a\rhd b})^{-1}(-(a\rhd b))+((L^\rhd_{a\rhd b})^{-1}(-(a\rhd b))\rhd (a+(a\rhd b)))$$ by Theorem \ref{q3}.
The result now follows from Lemma \ref{jie}.
\end{proof}
\begin{lemma}\label{q5}
Let $(T,+,\circ)$ be a dual weak left brace. Define a binary operation $``\rhd"$ on $T$ as follows: For all $a,b\in T$,
$a\rhd b=-a+(a\circ b).$
Then $(T,+,\rhd)$ is a strong post Clifford semigroup. If $\psi$ is a dual weak left brace homomorphism from $(T,+,\circ)$ to $(S,+,\circ)$, then $\psi$ is a post Clifford semigroup homomorphism from $(T,+,\rhd)$ to $(S,+,\rhd)$.
\end{lemma}
\begin{proof} Define $\lambda^T_a: T\rightarrow T, b\mapsto -a+a\,\circ\, b$ for all $a\in T$ and $$\lambda^T: (T, \circ)\rightarrow \End (T,+), \,\, a\mapsto \lambda^T_a.$$
By Lemma \ref{wangkangcccc}, $\lambda^T$ is a semigroup homomorphism.
Let $a,b,c\in T$.  Then $a\rhd(b+c)=\lambda_{a}^{T}(b+c)=\lambda_{a}^{T}(b)+\lambda_{a}^{T}(c).$
This gives (P1). By Lemma \ref{wang5}, we have $a^0+(a\circ b)=a^0\circ a\circ b=a\circ b$ and so
\begin{equation*}
\begin{aligned}
&(a+(a\rhd b))\rhd c=(a+\lambda_{a}^{T}(b))\rhd c=(a-a+a\circ b)\rhd c=(a^{0}+a\circ b)\rhd c\\
&=(a^{0}\circ a\circ b)\rhd c=(a\circ b)\rhd c=\lambda_{a\circ b}^{T}(c)
=\lambda_{a}^{T}\lambda_{b}^{T}(c)=a\rhd(b\rhd c).
\end{aligned}
\end{equation*}
This shows (P2) holds. Again by Lemma \ref{wang5}, we  obtain $a^{0}+(a\circ b^{0})=a^{0}\circ a\circ b^{0}=a\circ b^{0}$ and
$b^{0}\circ a = b^0\circ b^{0}\circ a=b^0+b^{0}\circ a=b^{0}-b^{0}+(b^{0}\circ a)$, whence
$$
a+(a\rhd b^{0})=a+\lambda_{a}^{T}(b^{0})=a-a+(a\circ b^{0})=a^{0}+(a\circ b^{0})
=a\circ b^{0}$$$$=b^{0}\circ a
=b^{0}-b^{0}+b^{0}\circ a=b^{0}+\lambda_{b^{0}}^{T}(a)
=b^{0}+(b^{0}\rhd a).
$$
Thus (P3) is true.  Denote  $H_{a}=\{x\in T\mid x^{0}=a^{0}\}$. Then $H_{a}=H_{a^{-1}}$. Assume that $b\in H_{a}$. Then $b^{0}=a^{0}$.
This implies that
$$
(L_{a}^{\rhd}(b))^{0} =(a\rhd b)^{0}=(-a+a\circ b)^{0}=(-a)^{0}+(a\circ b)^{0}
 $$$$=a^{0}+a^{0}\circ b^{0}=a^{0}+a^{0}+b^{0}=a^{0}+a^{0}+a^{0}
 =a^{0}.$$
Thus $L_{a}^{\rhd}(b)\in H_{a}$. Dually, we have $L_{a^{-1}}^{\rhd}(b)\in H_{a}$.  Moreover,
$$
L_{a}^{\rhd}L_{a^{-1}}^{\rhd}(b)=\lambda_{a}^{T}\lambda_{a^{-1}}^{T}(b)=\lambda_{a\circ a^{-1}}^{T}(b)=\lambda_{a^{0}}^{T}(b)
=-a^{0}+a^{0}\circ b$$$$=a^{0}+a^{0}\circ b=a^{0}\circ a^{0}\circ b
=a^{0}\circ b=b^{0}\circ b=b, $$
$$L_{a^{-1}}^{\rhd}L_{a}^{\rhd}(b)=\lambda_{a^{-1}}^{T}\lambda_{a}^{T}(b)=\lambda_{a^{0}}^{T}(b)=b.$$
This gives that $(L_{a}^{\rhd}|_{H_{a}})^{-1}=L_{a^{-1}}^{\rhd}|_{H_{a}}$ and so $L_{a}^{\rhd}|_{H_{a}}\in {\rm Aut}(H_{a},+)$ by (P1). Thus (P4) is valid.
(P5) follows from the fact that $a^{0}+(a\rhd b)=a^{0}-a+(a\circ b)=-a+(a\circ b)=a\rhd b.$ Then $(T,+,\rhd)$ is a strong post Clifford semigroup. Finally,
observe that $$\psi(a\rhd b)=\psi(-a+a\circ b)=\psi(-a)+\psi(a\circ b)=-\psi(a)+(\psi(a)\circ \psi(b))=\psi(a)\rhd \psi(b),$$ it follows that
$\psi$ is a post Clifford semigroup homomorphism.
\end{proof}

\begin{theorem}The category $\mathbb{SPCS}$ of strong post Clifford semigroups together with post Clifford semigroup homomorphisms is isomorphic to the category
$\mathbb{DWLB}$ of dual weak left braces together with dual weak left brace homomorphisms.
\end{theorem}
\begin{proof}
Define $\mathcal{F}: \mathbb{SPCS}\rightarrow \mathbb{DWLB}$ as follows: For all $(T,+,\rhd), (S,+,\rhd)\in {\rm Obj}(\mathbb{SPCS})$ and $\phi\in {\rm Hom}((T,+,\rhd), (S,+,\rhd))$, let $\mathcal{F}(T,+,\rhd)=(T,+,\circ)$ and $\mathcal{F}(\phi)=\phi$.
Define $\mathcal{G}: \mathbb{DWLB}\rightarrow \mathbb{SPCS}$ as follows:  For all $(T,+,\circ), (S,+,\circ)\in {\rm Obj}(\mathbb{DWLB})$ and $\phi\in {\rm Hom}((T,+,\circ), (S,+,\circ))$, let $\mathcal{G}(T,+,\circ)=(T,+,\rhd)$ and  $\mathcal{G}(\phi)=\phi$.
By Lemmas \ref{q4} and \ref{q5}, $\mathcal{F}$ and $\mathcal{G}$ are functors.

Let $(T,+,\rhd)$ be a post Clifford semigroup and denote $\mathcal{F}(T,+,\rhd)=(T,+,\circ)$ and $\mathcal{GF}(T,+,\rhd)=(T,+,\rhd')$. Let $a,b\in T$. Then
$$a\rhd'b=-a+a\circ b=-a+a+a\rhd b=a^{0}+a\rhd b=a\rhd b.$$
This shows that $\rhd'$ is equal to $\rhd$. Thus $\mathcal{GF}=\rm{id}_{\mathbb{SPCS}}$.

Let $(T,+,\circ)$ be a dual weak left brace and $\mathcal{G}(T,+,\circ)=(T,+,\rhd)$ and $\mathcal{FG}(T,+,\circ)=(T,+,\circ')$. Let $a,b\in T$. Then
$$a\circ'b=a+(a\rhd b)=a-a+a\circ b=a^{0}+a\circ b=a^{0}\circ a\circ b=a\circ b.$$
This shows that $\circ$ is equal to $\circ'$. Thus $\mathcal{FG}={\rm id}_{\mathbb{DWLB}}$.
\end{proof}

\section{Post Clifford semigroups and relative Rota-Baxter Clifford semigroups}
In this section, we introduce and study relative Rota-Baxter operators on Clifford semigroups and establish the relation between post Clifford semigroups and relative Rota-Baxter Clifford semigroups. Moreover, we also consider the substructures and quotient structures of relative Rota-Baxter Clifford semigroups.
\begin{defn}\label{wang11}
Let $(T,+)$ and $(S, +)$ be two Clifford semigroups,
$$\varphi: (S, +)\rightarrow (\End(T, +), \cdot),\,\, a\mapsto \varphi_a$$ be a semigroup homomorphism and $R: (T,+)\rightarrow (S,+)$ be a map. Then the quadruple
$((T,+), (S,+), \varphi, R)$ is called a {\em relative Rota--Baxter Clifford semigroup}  if the following conditions hold:  For all $a,b\in T$,
\begin{itemize}
\item[\rm (C1)] $R(a)+R(b)=R(a+\varphi_{R(a)}(b))$.
\item[\rm (C2)] $\varphi_{R(a)^{0}}(a)=a$.
\item[\rm (C3)] $R(a^{0})=R(a)^{0}$.
\item[\rm (C4)] $a+\varphi_{R(a)}(b^{0})=b^{0}+\varphi_{R(b^{0})}(a)$.
\end{itemize}
In this case, the map $R$ is referred to as the {\em relative Rota--Baxter operator} on $(T,+)$. If $R$ in additional  is a bijection, then $((T,+), (S,+), \varphi, R)$ is called a {\em bijective  relative Rota--Baxter Clifford semigroup.}
\end{defn}
\begin{lemma}\label{wang2}
Let $((T,+), (S,+), \varphi, R)$ be a relative Rota--Baxter Clifford semigroup with $a,b\in T$ and $x,y\in S$.
\begin{itemize}
\item[(1)]$\varphi_{x}\varphi_{y}(a)=\varphi_{x+y}(a)$ and $\varphi_{x}(a+b)=\varphi_{x}(a)+\varphi_{x}(b)$.
\item[(2)]$\varphi_{R(a)^{0}}(a^{0})=\varphi_{R(a^{0})}(a^{0})=a^{0}$.
\item[(3)]$-R(a)=R(\varphi_{-R(a)}(-a))$.
\item[(4)] $(\Ima R, +)$ is a Clifford subsemigroup of $(S,+)$.
\item[(5)] If $a\in E(T, +)$, then $R(a)\in E(S, +)$.
\end{itemize}
\end{lemma}
\begin{proof}
(1) This follows from the facts that $\varphi$ is a semigroup homomorphism and $\varphi_a\in \End(T, +)$.

(2) Replacing $a$ by $a^{0}$ in (C2) and then using (C3), we can obtain (2).

(3) Observe that
$$
R(a)+R(\varphi_{-R(a)}(-a))=R(a+\varphi_{R(a)}\varphi_{-R(a)}(-a))=R(a+\varphi_{R(a)^{0}}(-a))
$$$$=R(a-\varphi_{R(a)^{0}}(a))=R(a-a)
=R(a^{0}), $$
$$
R(a)+R(\varphi_{-R(a)}(-a))+R(a)=R(a^{0})+R(a)=R(a^{0}+\varphi_{R(a^{0})}(a))$$$$
=R(a^{0}+\varphi_{R(a)^{0}}(a))=R(a^{0}+a)=R(a),$$
\begin{equation*}
\begin{aligned}
&R(\varphi_{-R(a)}(-a))+R(a)+R(\varphi_{-R(a)}(-a)) \\
&=R(\varphi_{-R(a)}(-a))+R(a^{0})=R(a^{0})+R(\varphi_{-R(a)}(-a)) \\
&=R(a^{0}+\varphi_{R(a^{0})}\varphi_{-R(a)}(-a))=R(a^{0}+\varphi_{-R(a)}(-a)) \\
&=R(\varphi_{R(a^{0})^{0}}(a^{0})+\varphi_{-R(a)}(-a))=R(\varphi_{R(a)^{0}}(a^{0})+\varphi_{-R(a)}(-a)) \\
&=R(\varphi_{-R(a)}\varphi_{R(a)}(a^{0})+\varphi_{-R(a)}(-a))=R(\varphi_{-R(a)}(\varphi_{R(a)}(a^{0})-a)) \\
&=R(\varphi_{-R(a)}(-(a+\varphi_{R(a)}(a^{0}))))=R(\varphi_{-R(a)}(-(a^{0}+\varphi_{R(a^{0})}(a)))) \\
&=R(\varphi_{-R(a)}(-(a^{0}+a)))=R(\varphi_{-R(a)}(-a)),
\end{aligned}
\end{equation*}
it follows  that $-R(a)=R(\varphi_{-R(a)}(-a))$.

 (4) This follows from (C1) and item (2) in the present lemma.

 (5) Since $a\in E(T, +)$, we have $a^{0}=a$, and so $$R(a)+R(a)=R(a+\varphi_{R(a)}(a))=R(a+\varphi_{R(a^{0})}(a))$$$$=R(a+\varphi_{R(a)^{0}}(a))=R(a+a)=R(a)\in E(S,+)$$ by (C1) and (C2).
\end{proof}

\begin{defn} A relative Rota--Baxter Clifford semigroup $((T,+), (S,+), \varphi, R)$ is called  {\em strong} if it in additional satisfies the following condition: For all $a,b\in T$,
\begin{itemize}
\item[\rm (C5)] $a^{0}+\varphi_{R(a)}(b)=\varphi_{R(a)}(b)$.
\end{itemize}
In this case, the map $R$ is referred to as the {\em strong relative Rota--Baxter operator on $(T,+)$}.
\end{defn}

\begin{exam}Let $(T,+)$ be a Clifford semigroup and $R: T\rightarrow T$ be a map. Define $$\phi_a: T\rightarrow T, x\mapsto a+x-a$$ for each $a\in T$.
It is easy to see that  $\phi_a\in \End(T,+)$ for all $a\in T$.  Moreover,
$$\phi: (T, +)\rightarrow (\End(T,+),\cdot),\,\, a\mapsto \phi_a$$  is a semigroup homomorphism. Suppose that $((T,+), (T,+), \phi, R)$ is a relative Rota--Baxter Clifford semigroup and $a,b\in T$. Then (C1) gives that $$R(a)+R(b)=R(a+\varphi_{R(a)}(b))=R(a+R(a)+b-R(a)),$$ and (C2) gives that $$a=\varphi_{R(a)^{0}}(a)=R(a)^{0}+a-R(a)^{0}=a+R(a)^{0}.$$ This implies that $R$ is a Rota--Baxter operator on $(T, +)$. Assume  (C5) also holds.
Take $b=R(a)^{0}$ in (C5). Then we have $$a^{0}+R(a)^{0}=a^{0}+(R(a)+R(a)^{0}-R(a))=a^{0}+\varphi_{R(a)}(R(a)^{0})=\varphi_{R(a)}(R(a)^{0})=R(a)^{0}.$$
This together with Lemma \ref{wang2} (2) and (C3) gives that $$a^{0}=\varphi_{R(a)^{0}}(a^{0})=R(a)^{0}+a^{0}+R(a)^{0}=a^{0}+R(a)^{0}=R(a^{0}).$$
Thus $R$ is also strong in this case by Lemma \ref{strong}.

Conversely, let $R$ be a Rota--Baxter operator on $(T, +)$ and $a,b\in T$. Then by the above two equations, (C1) and (C2) are true. Lemma \ref{cwang} (1) gives (C3). Finally, we have  $$a+\varphi_{R(a)}(b^{0})=a+R(a)+b^{0} -R(a)=a+R(a)-R(a)+b^{0}=a+R(a)^{0}+b^{0}=a+b^{0} $$ by (C2).
By (C3), we have $R(b^{0})=R(b)^{0}\in E(T, +)$, and so  $R(b^{0})^{0}=R(b)^{00}=R(b)^{0}$. This together with (C2) implies that
$$b^{0}+\varphi_{R(b^{0})}(a)=b^{0}+R(b^{0})+a-R(b^{0})= b^{0}+R(b^{0})+a=b^{0}+R(b^{0})^{0}+a=b^{0}+a=a+b^{0}.$$
Thus (C4) holds. Therefore $((T,+), (T,+), \phi, R)$ is a  relative Rota--Baxter Clifford semigroup.
Assume that $R$ is also strong. Then $R(a^{0})=a^{0}$ by Lemma \ref{strong}, and so
$$a^{0}+\varphi_{R(a)}(b)=a^{0}+R(a)+b-R(a)=R(a^{0})+R(a)+b-R(a)$$$$=R(a)^{0}+R(a)+b-R(a)=R(a)+b-R(a)=\varphi_{R(a)}(b).$$
This gives (C5). Thus $(T, T, \phi, R)$ is also strong in this case.
\end{exam}
\begin{exam}Recall from \cite{l2} that a {\em relative Rota--Baxter group} is a quadruple $((H,+),\\ (G,+),\phi, R),$ where $(H,+)$ and $(G,+)$ are groups (not necessarily commutative), $\phi: G\rightarrow \Aut(H,+)$  a group homomorphism (where $\phi(g)$ is denoted by $\phi_g$ for all $g\in G$) and $R: H \rightarrow G$ is a map satisfying the condition
\begin{equation}\label{rbgroup}
R(h_1)+R(h_2)=R(h_1+\phi_{R(h_1)}(h_2))
\end{equation} for all $h_1, h_2\in H$. Obviously, $(H,+) $ and $(G,+)$ are Clifford semigroups and for all $h\in H$ and $g\in G$,  $h^{0}$  and $g^{0}$ are the identities $0_H$ and $0_G$ in $H$ and $G$, respectively. Clearly, (C1) is exactly (\ref{rbgroup}). Since $\phi$ is a group homomorphism and $\phi_g\in \Aut(H)$ for all $g\in G$, it follows that $\phi_{0_G}$ is the identity map on $H$ and $\phi_g(0_H)=0_H$ for all $g\in G$.
Let $a,b\in H$. Then $\phi_{R(a)^{0}}(a)=\phi_{0_G}(a)=a$, and so (C2) is valid. Take $h_1=h_2=0_H$ in (\ref{rbgroup}). Then  $$R(0_H)+R(0_H)=R(0_H+ \phi_{R(0_H)}(0_H))=R(0_H+ 0_H)=R(0_H),$$ which implies that $R(a^{0})=R(0_H)=0_G=R(a)^{0}$. This gives (C3). By similar argument, we can obtain (C4). The condition (C5) is trivially satisfied.
This yields that $((H,+), (G,+), \phi, R)$ forms a strong relative  Rota--Baxter Clifford semigroup. Thus, relative Rota--Baxter groups are strong relative  Rota--Baxter Clifford semigroups.
\end{exam}
\begin{prop}\label{5}Let $(T, S,\phi, R)$   be a relative Rota--Baxter Clifford semigroup with $\varphi\in \Aut (T,+)$, $\psi\in \End (S,+)$  and
$\varphi^{-1}\phi_{x}\varphi=\phi_{\psi(x)}$ for all $x\in S$. Then
$(T,S,\phi,\psi R\varphi)$ is also a relative Rota--Baxter Clifford semigroup. In particular,  if $\varphi$ lies in the center of $ \End (T,+)$, then $(T,S,\phi, R\varphi)$ is a relative Rota--Baxter Clifford semigroup.
\end{prop}
\begin{proof}
Let $a,b\in T$. Then
$$\psi R\varphi(a^{0})=\psi R(\varphi(a)^{0})=\psi (R(\varphi(a))^{0})=\psi R\varphi(a)^{0},$$
$$\phi_{\psi R\varphi(a)^{0}}(a)=\phi_{\psi(R(\varphi(a))^{0})}(a)
=\varphi^{-1}\phi_{R(\varphi(a))^{0}}\varphi(a)=\varphi^{-1}\varphi(a)=a,$$
\begin{equation*}
\begin{aligned}
&\psi R\varphi(a)+\psi R\varphi(b)=\psi(R\varphi(a)+R\varphi(b))=\psi(R(\varphi(a)+\phi_{R\varphi(a)}\varphi(b)))\\
&=\psi R(\varphi(a)+\phi_{R\varphi(a)}\varphi(b))=\psi R\varphi(a+\varphi^{-1}\phi_{R\varphi(a)}\varphi(b)) =\psi R\varphi(a+\phi_{\psi R\varphi(a)}\varphi(b)),
\end{aligned}
\end{equation*}
\begin{equation*}
\begin{aligned}
a+\phi_{\psi R\varphi(a)}(b^{0})&=a+\varphi^{-1}\phi_{R\varphi(a)}\varphi(b^{0})
=\varphi^{-1}(\varphi(a)+\phi_{R\varphi(a)}\varphi(b^{0}))\\
&=\varphi^{-1}(\varphi(b)^{0}+\phi_{R(\varphi(b)^{0})}\varphi(a))
=\varphi^{-1}(\varphi(b)^{0}+\phi_{R\varphi (b^{0})}\varphi(a))\\
&=b^{0}+\varphi^{-1}\phi_{R\varphi(b^{0})}\varphi(a)=b^{0}+\phi_{\psi R\varphi(b^{0})}\varphi(a).
\end{aligned}
\end{equation*}
This shows   (C1)--(C4) are all true, and hence the desired result follows.
\end{proof}

\begin{lemma}\label{h}
Let $((T,+), (S,+), \varphi, R)$  be a relative Rota--Baxter Clifford semigroup. Define  a binary operation $\circ_R$ as follows:
$$a\circ_{R}b=a+\varphi_{R(a)}(b) \mbox { for all } a,b\in T.$$
Then $(T, \circ_{R})$ is a Clifford semigroup and the inverse  of $a$ in $(T, \circ_R)$ is $a^{-1}=\varphi_{-R(a)}(-a)$  for all $a\in T$.  Moreover,
the map $R: (T, \circ_{R})\rightarrow (S,+)$ is a semigroup homomorphism.
\end{lemma}
\begin{proof}The operation $\circ_R$ is well-defined obviously.
Let $a,b\in T$.
\begin{equation*}
\begin{aligned}
&(a\circ_{R}b)\circ_{R}c=(a+\varphi_{R(a)}(b))\circ_{R}c =a+\varphi_{R(a)}(b)+\varphi_{R(a+\varphi_{R(a)}(b))}(c) \\
&=a+\varphi_{R(a)}(b)+\varphi_{R(a)+R(b)}(c) =a+\varphi_{R(a)}(b)+\varphi_{R(a)}\varphi_{R(b)}(c) \\
&=a+\varphi_{R(a)}(b+\varphi_{R(b)}(c)) =a+\varphi_{R(a)}(b\circ_{R}c)  =a\circ_{R}(b\circ_{R}c).
\end{aligned}
\end{equation*}
This gives that $(T,\circ_{R})$ is a semigroup. Denote $x=\varphi_{-R(a)}(-a)\in T$. Then
\begin{equation*}
\begin{aligned}
&a\circ_{R}x\circ_{R}a=a\circ_{R}\varphi_{-R(a)}(-a)\circ_{R}a=(a+\varphi_{R(a)}\varphi_{-R(a)}(-a))\circ_{R}a \\
&=(a+\varphi_{R(a)^{0}}(-a))\circ_{R}a=(a-\varphi_{R(a)^{0}}(a))\circ_{R}a \\
&=(a-a)\circ_{R}a=a^{0}\circ_{R}a =a^{0}+\varphi_{R(a)^{0}}(a)=a^{0}+a=a,
\end{aligned}
\end{equation*}
This together with Lemma \ref{wang2} (3) implies that
\begin{equation*}
\begin{aligned}
&x\circ_{R}a\circ_{R}x=x\circ_{R}a^{0}=\varphi_{-R(a)}(-a)\circ_{R}a^{0}  =\varphi_{-R(a)}(-a)+\varphi_{R(\varphi_{-R(a)}(-a))}(a^{0})\\  &=\varphi_{-R(a)}(-a)+\varphi_{-R(a)}(a^{0})  =\varphi_{-R(a)}(-a+a^{0}) =\varphi_{-R(a)}(-a)=x.
\end{aligned}
\end{equation*}
So $(T,\circ_{R})$ is a regular semigroup. Furthermore, let $e\in E(T,\circ_{R})$. Then
$$R(e)=R(e\circ_{R}e)=R(e+\varphi_{R(e)}(e))=R(e)+R(e),$$ and so $R(e)\in E(T,+).$ This implies that
$$e=e\circ_{R}e=e+\varphi_{R(e)}(e)=e+\varphi_{R(e)^{0}}(e)=e+e\in E(T,+).$$
Therefore $E(T,\circ_{R})\subseteq E(T,+)$.  So we can let $e=b^{0}$ for some $b\in T$. Then
$$a\circ_{R}e=a\circ_{R}b^{0}=a+\varphi_{R(a)}(b^{0})=b^{0}+\varphi_{R(b^{0})}(a)=e+\varphi_{R(e)}(a)=e\circ_{R}a.$$
We have shown that each idempotent in  $(T,\circ_{R})$ lies in the center of $(T,\circ_{R})$.  Thus $(T,\circ_{R})$ is a Clifford semigroup by Lemma \ref{Clifford}.
Moreover,  $a^{-1}=\varphi_{-R(a)}(-a)$ is the inverse of $a$ in $(T, \circ_R)$ by the above statements.
Obviously,
$R(a\circ_{R}b)=R(a+\varphi_{R(a)}(b))=R(a)+R(b).$
So $R$ is a semigroup homomorphism from $(T, \circ_{R})$ to $(S,+)$.
\end{proof}
\begin{defn}Let $((T,+),(S,+),\varphi,R)$   be a relative Rota--Baxter Clifford semigroup. The Clifford semigroup $(T, \circ_R)$ obtained in Lemma \ref{h} is called the
{\em descendent Clifford semigroup} of $R$.
\end{defn}

\begin{coro}\label{kangweiccc}
Let $((T,+),(S,+),\varphi,R)$    be a relative Rota--Baxter Clifford semigroup and $a\in S$. Denote $H_a=\{x\in S\mid x^{0}=a^{0}\}$. Then
$\varphi_{R(a)}(a^{0})=a^{0}$, $\varphi_{R(a)}(H_a)\subseteq H_a$ and $\varphi_{R(a)}|_{H_a}$ is an isomorphism on $H_a$ and $(\varphi_{R(a)}|_{H_a})^{-1}=\varphi_{-R(a)}|_{H_a}$.
\end{coro}
\begin{proof} By Lemma \ref{h},  $(T, \circ_{R})$ forms a Clifford semigroup and $a^{-1}=\varphi_{-R(a)}(-a)$.  This implies that
$$a^{0}=a-a=a\circ_R  a^{-1}=a^{-1}\circ_R a=\varphi_{-R(a)}(-a)+\varphi_{R(\varphi_{-R(a)}(-a))}(a)$$$$=\varphi_{-R(a)}(-a)+\varphi_{-R(a)}(a)=\varphi_{-R(a)}(a^{0}).$$
So $\varphi_{R(a)}(a^{0})=\varphi_{R(a)}(\varphi_{-R(a)}(a^{0}))=\varphi_{R(a)^{0}}(a^{0})=a^{0}.$ If $x\in H_a$, then $x^{0}=a^{0}$, and so
$\varphi_{R(a)}(x)^{0}=\varphi_{R(a)}(x^{0})=\varphi_{R(a)}(a^{0})=a^{0}$. This shows that $\varphi_{R(a)}(x)\in H_a$. Thus $\varphi_{R(a)}(H_a)\subseteq H_a$.
Similarly, $\varphi_{-R(a)}(H_a)\subseteq H_a$. Moreover, for all $x\in H_a$, $$\varphi_{R(a)}\varphi_{-R(a)}(x)=\varphi_{-R(a)}\varphi_{R(a)}(x)=\varphi_{R(a)^{0}}(x)=\varphi_{R(a^{0})}(x)=\varphi_{R(x^{0})}(x)=x.$$ Therefore, $\varphi_{R(a)}|_{H_a}$ is an automorphism on $H_a$.
\end{proof}
To give more characterizations of relative Rora-Baxter Clifford semigroups, we need a construction method of inverse semigroups  which  is known as  {\em the $\lambda$-semidirect products} in literature (see \cite{lawson} for details). The following lemma is a special case of \cite[Proposition 5.3.1]{lawson}, but we also give a proof for  completeness.
\begin{lemma}\label{wang3} Let $(T,+)$ and $(S, +)$ be  Clifford semigroups and
$$\varphi: (S, +)\rightarrow (\End(T, +), \cdot),\,\, a\mapsto \varphi_a$$ be a semigroup homomorphism. Define a multiplication on
$$M=M(S,T, \varphi)=\{(x,a)\in S\times T\mid \varphi_{x^{0}}(a)=a\}$$ as follows: For all $(x,a), (y,b)\in M$,
$$(x,a)+(y,b)=(x+y, \varphi_{y^{0}}(a)+\varphi_x(b)).$$ Then $(M +)$ forms an inverse semigroup with  $$E(M,+)=\{(x,a)\in M\mid x\in E(S, +), a\in E(T, +)\}$$
and $-(x,a)=(-x, \varphi_{-x}(-a))$ for all $(x,a)\in M$.
\end{lemma}
\begin{proof} Let $x\in S$ and $a\in T$. Then $\varphi_{x^{0}}(\varphi_{x}(a))=\varphi_{x^{0}+x}(a)=\varphi_{x}(a).$
This implies that $(x,\varphi_{x}(a))\in M$, and so $M$ is nonempty. Let $(x,a),(y,b), (z, c)\in M$. Then $\varphi_{x^{0}}(a)=a$, $\varphi_{y^{0}}(b)=b$ and  $\varphi_{z^{0}}(c)=c$. On one hand,
$$
\varphi_{(x+y)^{0}}(\varphi_{y^{0}}(a)+\varphi_{x}(b))=\varphi_{x^{0}+y^{0}}
\varphi_{y^{0}}(a)+\varphi_{x^{0}+y^{0}}\varphi_{x}(b)
=\varphi_{x^{0}+y^{0}+y^{0}}(a)+\varphi_{x^{0}+y^{0}+x}(b)
$$$$=\varphi_{x^{0}+y^{0}}(a)+\varphi_{x+y^{0}}(b)
=\varphi_{x^{0}+y^{0}}(a)+\varphi_{x}\varphi_{y^{0}}(b)
=\varphi_{y^{0}}\varphi_{x^{0}}(a)+\varphi_{x}(b)
=\varphi_{y^{0}}(a)+\varphi_{x}(b).
$$
This implies that $(x+y, \varphi_{y^{0}}(a)+\varphi_{x}(b))\in M$,  and so the above binary operation $+$ is well-defined.
On the other hand,
\begin{equation*}
\begin{aligned}
&((x,a)+(y,b))+(z,c)=(x+y, \varphi_{y^{0}}(a)+\varphi_{x}(b))+(z,c)\\
&=((x+y)+z, \varphi_{z^{0}}(\varphi_{y^{0}}(a)+\varphi_{x}(b))+\varphi_{x+y}(c))\\
&=(x+y+z, \varphi_{y^{0}+z^{0}}(a)+\varphi_{x+z^{0}}(b)+\varphi_{x+y}(c)),
\end{aligned}
\end{equation*}
\begin{equation*}
\begin{aligned}
&(x,a)+((y,b)+(z,c))=(x,a)+(y+z, \varphi_{z^{0}}(b)+\varphi_{y}(c))\\
&=(x+(y+z), \varphi_{y^{0}+z^{0}}(a)+\varphi_{x}(\varphi_{z^{0}}(b)+\varphi_{y}(c)))\\
&=(x+y+z, \varphi_{y^{0}+z^{0}}(a)+\varphi_{x+z^{0}}(b)+\varphi_{x+y}(c)) =((x,a)+(y,b))+(z,c).
\end{aligned}
\end{equation*}
This shows   $(M,+)$ forms a semigroup. It is easy to see that  $(-x,\varphi_{-x}(-a))\in M$. Moreover,
$$
(x,a)+(-x,\varphi_{-x}(-a))
=(x-x, \varphi_{(-x)^{0}}(a)+\varphi_{x}(\varphi_{-x}(-a)))
$$$$=(x^{0}, \varphi_{x^{0}}(a)+\varphi_{x^{0}}(-a))
=(x^{0}, \varphi_{x^{0}}(a^{0})),
$$
$$
(x^{0}, \varphi_{x^{0}}(a^{0}))+(x,a)
=(x^{0}+x, \varphi_{x^{0}}(\varphi_{x^{0}}(a^{0}))+\varphi_{x^{0}}(a))
$$$$=(x, \varphi_{x^{0}}(a^{0})+\varphi_{x^{0}}(a))
=(x, \varphi_{x^{0}}(a))=(x,a),$$
Since  $$(x,a)+(x,a)=(x+x, \varphi_{x^{0}}(a)+\varphi_{x}(a))=(x+x,a+\varphi_{x}(a)),$$  it follows that
$$
(x,a)+(x,a)=(x,a)
\Longleftrightarrow x+x=x \mbox{ and } a+\varphi_{x}(a)=a
$$$$\Longleftrightarrow x+x=x\mbox{ and } a+\varphi_{x^{0}}(a)=a
\Longleftrightarrow x+x=x\mbox{ and } a+a=a.
$$
So $$E(M,+)=\{(x,a)\in M \mid x\in E(S,+), a\in E(T,+)\}.$$ Let $(x,a),(y,b)\in E(M,+)$. Then
$$(x,a)+(y,b)=(x+y, \varphi_{y^{0}}(a)+\varphi_{x}(b))=(x+y, \varphi_{y}(a)+\varphi_{x}(b)).$$
Similarly, $(y,b)+(x,a)=(y+x, \varphi_{x}(b)+\varphi_{y}(a))$. Since $a,b\in E(T)$, we have $\varphi_{x}(b),\varphi_{y}(a)\in E(T,+)$, and so $\varphi_{x}(b)+\varphi_{y}(a)=\varphi_{y}(a)+\varphi_{x}(b)$. This implies that $(x,a)+(y,b)=(y,b)+(x,a)$. Thus $(M,+)$ is an inverse semigroup by Lemma \ref{inverse}.
\end{proof}
\begin{theorem}Let $(T,+)$ and $(S, +)$ be  two Clifford semigroups,
$$\varphi: (S, +)\rightarrow (\End(T, +), \cdot),\,\, a\mapsto \varphi_a$$ be a semigroup homomorphism and $R:T\rightarrow S$ be a map. Then the quadruple
$((T,+),(S,+),\\\varphi,R)$  is  a  relative Rota--Baxter Clifford semigroup if and only if
${\rm Gr}R=\{(R(a),a)\mid a\in T\}$ is a Clifford subsemigroup of $(M,+)=(M(S, T, \varphi),+)$.  In this case, if $(T, \circ_R)$ is the  descendent Clifford semigroup of $R$, then
$$\phi: (T, \circ_R)\rightarrow ({\rm Gr}R,+),\,\, a\mapsto (R(a),a)$$
is a  semigroup isomorphism.
\end{theorem}

\begin{proof} Assume that $(T, S, \varphi, R)$ is  a  relative Rota--Baxter Clifford semigroup and $a,b,c\in T$.
Firstly, Since $\varphi_{R(a)^{0}}(a)=a$ by (C2), we have $(R(a),a)\in M$. This implies that ${\rm Gr}R\subseteq M$. Secondly,
\begin{equation*}
\begin{aligned}
&(R(a),a)+ (R(b),b)=(R(a)+R(b),\varphi_{R(b)^{0}}(a)+\varphi_{R(a)}(b)) \\
&=(R(a)+R(b),\varphi_{R(b)^{0}}(a)+\varphi_{R(a)+R(b)^{0}}(b)) =(R(a)+R(b),\varphi_{R(b)^{0}}(a+\varphi_{R(a)}(b))),
\end{aligned}
\end{equation*}
\begin{equation*}
\begin{aligned}
&\varphi_{R(b)^{0}}(a+\varphi_{R(a)}(b))=\varphi_{R(b)^{0}}(a+\varphi_{R(a)}(b^{0})+\varphi_{R(a)}(b)) \\
&=\varphi_{R(b)^{0}}(b^{0}+\varphi_{R(b^{0})}(a)+\varphi_{R(a)}(b))  =\varphi_{R(b)^{0}}(b^{0})+\varphi_{R(b)^{0}}(a)+\varphi_{R(a)+R(b)^{0}}(b) \\
&=b^{0}+\varphi_{R(b)^{0}}(a)+\varphi_{R(a)}(b)  =a+\varphi_{R(a)}(b^{0})+\varphi_{R(a)}(b)  =a+\varphi_{R(a)}(b).
\end{aligned}
\end{equation*}
This implies that $$(R(a),a)+ (R(b),b)=(R(a)+R(b),a+\varphi_{R(a)}(b)).$$ Since $R(a)+R(b)=R(a+\varphi_{R(a)}(b))$ by (C1), we have $(R(a),a)+ (R(b),b)\in {\rm Gr}R$. Observe that
$$-(R(a),a)=(-R(a),\varphi_{-R(a)}(-a))\text{ and } -R(a)=R(\varphi_{-R(a)}(-a)).$$
it follows that $-(R(a),a)\in {\rm Gr}R$.  This shows that $({\rm Gr}R, +)$ is an inverse subsemigroup of $(M,+)$.
To prove $({\rm Gr}R, +)$ is a Clifford semigroup,  we first observe that $$E({\rm Gr}R, +)=\{(R(a),a)\mid a\in E(T, +)\}
=\{(R(b^{0}), b^{0})\mid b\in T\}$$ by Lemma \ref{wang2} (4) and Lemma \ref{wang3}.  Since
\begin{equation*}
\begin{aligned}
&(R(c),c)+ (R(b^{0}),b^{0})=(R(c)+R(b^{0}),\varphi_{R(b^{0})}(c)+\varphi_{R(c)}(b^{0})) \\
&=(R(c)+R(b^{0}),\varphi_{R(b^{0})}(c)+\varphi_{R(c)+R(b^{0})}(b^{0})) =(R(c)+R(b^{0}),\varphi_{R(b^{0})}(c+\varphi_{R(c)}(b^{0}))),
\end{aligned}
\end{equation*}
\begin{equation*}
\begin{aligned}
(R(b^{0}),b^{0})+ (R(c),c)&=(R(b^{0})+R(c),\varphi_{R(b^{0})+R(c)^{0}}(b^{0})+\varphi_{R(b^{0})}(c))\\
&=(R(c)+R(b^{0}),\varphi_{R(b^{0})}(\varphi_{R(c)^{0}}(b^{0})+\varphi_{R(b^{0})}(c))),
\end{aligned}
\end{equation*}
\begin{equation*}
\begin{aligned}
&\varphi_{R(b^{0})}(\varphi_{R(c)^{0}}(b^{0})+\varphi_{R(b^{0})}(c))
=\varphi_{R(b^{0})}(\varphi_{R(c)^{0}}(b^{0})+\varphi_{R(b^{0})+R(c)^{0}}(c))\\
&=\varphi_{R(b^{0})+R(c^{0})}(b^{0}+\varphi_{R(b^{0})}(c)) =\varphi_{R(b^{0})+R(c^{0})}(c+\varphi_{R(c)}(b^{0}))\\
&=\varphi_{R(b^{0})}(\varphi_{R(c^{0})}(c)+\varphi_{R(c^{0})+R(c)}(b^{0})) =\varphi_{R(b^{0})}(c+\varphi_{R(c)}(b^{0})),
\end{aligned}
\end{equation*}
we have $$(R(c),c)+ (R(b^{0}),b^{0})=(R(b^{0}),b^{0})+ (R(c),c).$$ This implies that each idempotent of $({\rm GrR}, +)$ lies in the center of  $({\rm GrR}, +)$, and so $({\rm GrR}, +)$ is a Clifford subsemigroup.

Conversely, assume that $({\rm GrR}, +)$ is a Clifford subsemigroup of $(M, +)$. Let $a,b\in T$. Then
$(R(a),a)\in {\rm Gr}R\subseteq M$ and so $\varphi_{R(a)^{0}}(a)=a.$
This implies that (C2) holds. Since
\begin{equation*}
\begin{aligned}
&(R(a),a)+ (R(b),b)=(R(a)+R(b),\varphi_{R(b)^{0}}(a)+\varphi_{R(a)}(b)) \\
&=(R(a)+R(b),\varphi_{R(b)^{0}}(a+\varphi_{R(a)}(b)))=(R(a)+R(b),a+\varphi_{R(a)}(b))\in {\rm  Gr}R,
\end{aligned}
\end{equation*}
and so $R(a)+R(b)=R(a+\varphi_{R(a)}(b))$. This gives (C1).   Since $(R(a),a)\in {\rm Gr}R$, we have $-(R(a),a)=(-R(a),\varphi_{-R(a)}(-a))\in  {\rm Gr}R$ and
\begin{equation*}
\begin{aligned}
(R(a),a)+ (-R(a),\varphi_{-R(a)}(-a))&=(R(a)^{0},\varphi_{(-R(a))^{0}}(a)+\varphi_{R(a)}\varphi_{-R(a)}(-a))\\
&=(R(a)^{0},a-a)=(R(a)^{0},a^{0})\in  {\rm Gr}R,
\end{aligned}
\end{equation*}
This gives $R(a)^{0}=R(a^{0})$. Thus (C3) is true.  Replacing $a$ by $a^{0}$ in (C2), we have
\begin{equation}\label{wang4}
a^{0}=\varphi_{R(a^{0})^{0}}(a^{0})=\varphi_{R(a)^{00}}(a^{0})=\varphi_{R(a)^{0}}(a^{0})
\end{equation} by (C3).

Finally, since $b\in T$, we have $(R(b^{0}),b^{0})\in E({\rm GrR}, +)$, and so
$$(R(a),a)+ (R(b^{0}),b^{0})=(R(a)+R(b^{0}),\varphi_{R(a)^{0}+R(b)^{0}}(a)+\varphi_{R(a)}(b^{0})),$$
$$(R(b^{0}),b^{0})+ (R(a),a)=(R(a)+R(b^{0}),\varphi_{R(b)^{0}+R(a)^{0}}(b^{0})+\varphi_{R(b)^{0}}(a)).$$
Observe that $({\rm GrR}, +)$ is a Clifford subsemigroup, we have
$$\varphi_{R(a)^{0}+R(b)^{0}}(a)+\varphi_{R(a)}(b^{0})=\varphi_{R(b)^{0}+R(a)^{0}}(b^{0})+\varphi_{R(b)^{0}}(a)).$$
By direct calculations using (C1)--(C3) and (\ref{wang4})), we have
\begin{equation*}
\begin{aligned}
&\varphi_{R(a)^{0}+R(b)^{0}}(a)+\varphi_{R(a)}(b^{0}) =\varphi_{R(a)^{0}}(\varphi_{R(b)^{0}}(a)+\varphi_{R(a)}(b^{0}))\\
&=\varphi_{R(a)^{0}}(\varphi_{R(b)^{0}}(a)+\varphi_{R(a)+R(b)^{0}}(b^{0}))
 =\varphi_{R(a)^{0}+R(b)^{0}}(a+\varphi_{R(a)}(b^{0}))\\
&=\varphi_{(R(a)+R(b)^{0})^{0}}(a+\varphi_{R(a)}(b^{0})) =\varphi_{R(a+\varphi_{R(a)}(b^{0}))^{0}}(a+\varphi_{R(a)}(b^{0})) =a+\varphi_{R(a)}(b^{0}),
\end{aligned}
\end{equation*}
and
\begin{equation*}
\begin{aligned}
&\varphi_{R(b)^{0}+R(a)^{0}}(b^{0})+\varphi_{R(b)^{0}}(a)
=\varphi_{R(b)^{0}}(\varphi_{R(a)^{0}}(b^{0})+\varphi_{R(b)^{0}}(a))\\
&=\varphi_{R(b)^{0}}(\varphi_{R(a)^{0}}(b^{0})+\varphi_{R(b)^{0}+R(a)^{0}}(a))\\
&=\varphi_{R(b)^{0}+R(a)^{0}}(b^{0}+\varphi_{R(b)^{0}}(a)) =\varphi_{(R(b)^{0}+R(a))^{0}}(b^{0}+\varphi_{R(b^{0})}(a))\\
&=\varphi_{R(b^{0}+\varphi_{R(b^{0})}(a))^{0}}(b^{0}+\varphi_{R(b^{0})}(a)) =b^{0}+\varphi_{R(b^{0})}(a).
\end{aligned}
\end{equation*}
This gives that $a+\varphi_{R(a)}(b^{0})=b^{0}+\varphi_{R(b^{0})}(a)$. Thus (C4) holds.

Obviously, the map $\phi$ in the theorem is an bijection. For all $a,b\in T$,
by Lemma \ref{h},
\begin{equation*}
\begin{aligned}
&\varphi_{a\circ_{R}b}=(R(a\circ_{R}b),a\circ_{R}b)=(R(a)+R(b),a+\varphi_{R(a)}(b))\\
&=(R(a)+R(b),\varphi_{R(b)^{0}}(a+\varphi_{R(a)}(b)))=(R(a)+R(b),\varphi_{R(b)^{0}+R(a)^{0}}(a)+\varphi_{R(a)}(b))\\
&=(R(a),a)+ (R(b),b)=\varphi(a)+ \varphi(b).
\end{aligned}
\end{equation*}
Thus $\varphi$ is a semigroup isomorphism.
\end{proof}

\begin{defn}\label{wang8} Let $((T,+),(S,+),\phi,R)$ and $((M,+),(N,+),\varphi,B)$ be two relative Rota-Baxter Clifford semigroups and $\psi: (T,+)\rightarrow (M, +)$ and $\eta: (S,+)\rightarrow (N,+)$ be semigroup homomorphisms (respectively, isomorphisms). Then $(\psi,\eta): (T,S,\phi,R)\rightarrow (M,N,\varphi,B)$  is called a {\em   Rota-Baxter  homomorphism (respectively, Rota-Baxter isomorphism)} if $\eta R=B\psi $ and $\psi\phi_{a}=\varphi_{\eta(a)}\psi$ for all $a\in S$.
\end{defn}
\begin{lemma}\label{liweia}
Let $(T,+,\rhd)$ be a post Clifford semigroup with the sub-adjacent Clifford semigroup $(T,\circ)$. Then $R(T)=((T,+),(T,\circ),L^{\rhd}, {\rm id}_{T})$ forms a relative Rota-Baxter Clifford semigroup. If $(T,+,\rhd)$ is strong, then $R(T)$ is also strong. Moreover, if $\psi$ is a post Clifford semigroup homomorphism from $(T,+,\rhd)$ to $(S,+,\rhd)$, then $(\psi,\psi)$ is a Rota-Baxter homomorphism from $R(T)$ to $R(S)$.
\end{lemma}
\begin{proof}
By Lemma \ref{q3} (2), $L^{\rhd}: (T,\circ)\rightarrow {\rm End}(T,+),a\mapsto L_{a}^{\rhd}$ is a semigroup homomorphism. Obviously, ${\rm id}_{T}:(T,+)\rightarrow (T,\circ),x\mapsto x$ is map. Let $a,b\in T$. Then
$${\rm id}_{T}(a)\circ {\rm id}_{T}(b)=a\circ b=a+(a\rhd b)=a+L_{a}^{\rhd}(b)={\rm id}_{T}(a+L_{{\rm id}_{T}(a)}^{\rhd}(b)),$$ which gives (C1).
Lemma \ref{q2} gives $L_{{\rm id}_{T}(a)^{0}}^{\rhd}(a)=L_{a^{0}}^{\rhd}(a)=a$ and
$${\rm id}_{T}(a^{0})=a^{0}=a\circ a^{-1}={\rm id}_{T}(a)\circ {\rm id}_{T}(a)^{-1}={\rm id}_{T}(a)^{0},$$ and so (C2) and (C3) hold.
(P3) implies that
$$a+L_{{\rm id}_{T}(a)}^{\rhd}(b^{0})=a+L_{a}^{\rhd}(b^{0})=a+(a\rhd b^{0})=b^{0}+(b^{0}\rhd a)=b^{0}+L_{{\rm id}_{T}(b^{0})}^{\rhd}(a),$$ which gives (C4).
If $(T,+,\rhd)$ is strong, then
$$a^{0}+L_{{\rm id}_{T}(a)}^{\rhd}(b)=a^{0}+(a\rhd b)=a\rhd b=L_{{\rm id}_{T}(a)}^{\rhd}(b),$$ which implies that $R(T)$ is also strong. Let $\psi$ be a post Clifford semigroup homomorphism from $(T,+,\rhd)$ to $(S,+,\rhd)$. Then $\psi$ preserves $``+"$ and $``\circ"$ by Theorem \ref{q3} (3). Obviously, $\psi{\rm id}_T={\rm id}_S\psi$.
Let $a,b\in T$. Then $$\psi L^{\rhd_T}_a(b)=\psi(a\rhd_T b)=\psi(a)\rhd_S \psi(b)= L^{\rhd_T}_{\psi(a)}(\psi(b))=L^{\rhd_T}_{\psi(a)}\psi(b).$$
This gives $\psi L^{\rhd_T}_a=L^{\rhd_T}_{\psi(a)}\psi$.
Thus $(\psi,\psi)$ is a Rota-Baxter homomorphism.
\end{proof}
\begin{lemma}\label{liweib}
Let $((T,+),(S,+),\varphi,R)$ be  a relative Rota-Baxter Clifford semigroup. Define a binary operation $``\rhd"$ on $T$ as follows: For all $a,b\in T$,
$a\rhd b=\varphi_{R(a)}(b)$. Then $(T,+,\rhd)$ is a post Clifford semigroup and $(L^\rhd_a|_{H_a})^{-1}=\varphi_{-R(a)}|_{H_a}$  for all $a\in T$. If $((T,+),(S,+),\varphi,R)$ is strong, then $(T,+,\rhd)$ is also strong.
Moreover, if $(\psi,\eta)$ is a Rota-Baxter homomorphism from $((T,+),(S,+),\varphi,R)$ to $((M,+),(N,+),\phi,B)$, then $\psi$ is
a post Clifford semigroup homomorphism from $(T,+,\rhd)$ to $(S,+,\rhd)$.
\end{lemma}
\begin{proof}
Let $a,b,c\in T$. Since $\varphi_{R(a)}\in {\rm End}(T,+)$, we have
$$a\rhd (b+c)=\varphi_{R(a)}(b+c)=\varphi_{R(a)}(b)+\varphi_{R(a)}(c)=(a\rhd b)+(a\rhd c).$$
So (P1) holds. Since $\varphi$ is a semigroup homomorphism, by (C1) we obtain that
$$(a+(a\rhd b))\rhd c=(a+\varphi_{R(a)}(b))\rhd c=\varphi_{R(a+\varphi_{R(a)}(b))}(c)$$$$
=\varphi_{R(a)+R(b)}(c)=\varphi_{R(a)}\varphi_{R(b)}(c)
=a\rhd (b\rhd c).$$ This proves (P2).
Moreover, (C4) gives that $$a+(a\rhd b^{0})=a+\varphi_{R(a)}(b^{0})=b^{0}+\varphi_{R(b^{0})}(a)=b^{0}+(b^{0}\rhd a),$$  and so (P3) is true.
Observe that $L^\rhd_a (b)=a\rhd b=\varphi_{R(a)}(b)$, it follows that $L^\rhd_a=\varphi_{R(a)}$ for all $a,b\in T$.  By Corollary \ref{kangweiccc}, (P4) holds and $(L^\rhd_a|_{H_a})^{-1}=\varphi_{-R(a)}|_{H_a}$  for  all $a\in T$.
If $((T,+),(S,+),\varphi,R)$ is strong, then $a^{0}+(a\rhd b)=a^{0}+\varphi_{R(a)}(b)=\varphi_{R(a)}(b)=a\rhd b.$ This shows (P5) and so $(T,+,\rhd)$ is also strong. Let $(\psi,\eta)$ be a Rota-Baxter homomorphism from $((T,+),(S,+),\varphi,R)$ to $((M,+),(N,+),\phi,B)$. Then $\eta R=B\psi$ and $\eta\varphi_a=\phi_{\eta(a)}\psi$ for all $a\in T$. This implies that $$\psi(a\rhd b) =\psi(\varphi_{R(a)}(b))=(\psi\varphi_{R(a)})(b)=(\phi_{\eta(R(a))}\psi)(b)$$$$=\phi_{\eta(R(a))}(\psi(b))=\phi_{(\eta R)(a)}(\psi(b))=\phi_{(B\psi)(a)}(\psi(b))=\phi_{B(\psi(a))}(\psi(b))=\psi(a)\rhd \psi(b).$$ This shows that $\psi$ is
a post Clifford semigroup homomorphism.
\end{proof}
\begin{coro} In Definition \ref{wang11}, in the presence of (C1)--(C3), (C4) holds if and only if
$a+\phi_{R(a)}(b^0)=a+b^0$ and $b^0+\phi_{R(b^0)}(a)=b^0 +a$ for all $a,b\in T$.
\end{coro}
\begin{proof}
By Theorem \ref{q3} and Lemma \ref{q4} and \ref{liweib}, $(T,+,\circ)$ forms a dual weak left brace, where $a\circ b=a+\varphi_{R(a)}(b)$ for all $a,b\in T$. In view of Lemma \ref{wang5},
$$a+b^{0}=a\circ_R b^{0}=a+\phi_{R(a)}(b^{0}),\,\,\,\,b^{0}+\phi_{R(b^{0})}(a)=b^{0}\circ_{R} a=b^{0} +a.$$This gives the necessity. The sufficiency is trivial.
\end{proof}

\begin{coro}\label{rotajie}Let $((T,+),(S,+),\varphi,R)$ be  a relative Rota-Baxter Clifford semigroup.  Then  the map $r: T\times T\rightarrow T\times T$ defined by
$$r(a,b)=(-a+a+\varphi_{R(a)}(b),\,\, \varphi_{-R(\varphi_{R(a)}(b))}(-\varphi_{R(a)}(b)+a+\varphi_{R(a)}(b))$$
$\mbox{ for all } a,b\in T$, is a solution of the Yang-Baxter equation.
\end{coro}
\begin{proof}
Define a binary operation $``\rhd"$ on $T$ as follows: For all $a,b\in T$, $a\rhd b=\varphi_{R(a)}(b)$. Then $(T,+,\rhd)$ is a post Clifford semigroup and $(L^\rhd_a|_{H_a})^{-1}=\varphi_{-R(a)}|_{H_a}$  for  all $a\in T$ by Lemma \ref{liweib}. Let $a,b\in T$. Observe that $-a+a+(a\rhd b)=-a+a+\varphi_{R(a)}(b)$. Denote $u=(L^\rhd_{a\rhd b}|_{H_{a\rhd b}})^{-1}(-(a\rhd b))$. Then $u=\varphi_{-R(a\rhd b)}(-(a\rhd b))$, whence  $R(u)=-R(a\rhd b)$ by Lemma \ref{wang2} (3).
This implies that  $$(L^\rhd_{a\rhd b}|_{H_{a\rhd b}})^{-1}(-(a\rhd b))+((L^\rhd_{a\rhd b}|_{H_{a\rhd b}})^{-1}(-(a\rhd b))\rhd (a+(a\rhd b)))=u+(u\rhd(a+(a\rhd b)))$$
$$=u+\varphi_{R(u)}(a+(a\rhd b))=\varphi_{-R(a\rhd b)}(-(a\rhd b))+\varphi_{-R(a\rhd b)}(a+(a\rhd b))\overset{\rm Lemma\, \ref{wang2}\, (1)}{=}$$$$\varphi_{-R(a\rhd b)}(-(a\rhd b)+a+(a\rhd b))=\varphi_{-R(\varphi_{R(a)}(b))}(-(\varphi_{R(a)}(b))+a+(\varphi_{R(a)}(b))).$$ Thus the result follows from Theorem \ref{postjie}.
\end{proof}

\begin{theorem}The category ${\mathbb{BRRBCS}}$ of bijective  relative Rota-Baxter Clifford semigroups together with   Rota--Baxter  homomorphisms is equivalent to the category ${\mathbb{PCS}}$ of post Clifford semigroups together with  post Clifford semigroups homomorphisms. As a consequence, the category   of bijective strong  relative Rota--Baxter Clifford semigroups   is equivalent to the category  of strong post Clifford semigroups.
\end{theorem}
\begin{proof}
Define ${\mathcal F}: \mathbb{BRRBCS}\rightarrow \mathbb{PCS}$ as follows:  For any $$(T,S,\phi, R), (M,N,\varphi,B)\in {\rm Obj}(\mathbb{BRRBCS}) \mbox{ and } (\psi,\eta)\in {\rm Hom}((T,S,\phi, R),(M,N,\varphi,B)),$$
$${\mathcal F}(T,S,\phi, R)=(T, +, \rhd_R),\,\,\,\,\, {\mathcal F}(\psi,\eta)=\psi:  (T, +, \rhd_R)\rightarrow  (M, +, \rhd_B).$$
On the other hand, define ${\mathcal G}:\mathbb{PCS}\rightarrow \mathbb{BRRBCS}$ as follows: For any  $(T, +, \rhd),  (S, +, \rhd)\in {\rm Obj}(\mathbb{PCS})$ and
$\psi\in {\rm Hom}((T, +, \rhd), (S, +, \rhd))$,
$${\mathcal G} (T, +, \rhd)=((T,+), (T,\circ), L^{\rhd_T}, {\rm id}_T),$$
$${\mathcal G}(\psi)=(\psi,\psi): ((T,+), (T,\circ), L^{\rhd_T},  {\rm id}_T)\rightarrow  ((S,+), (S,\circ), L^{\rhd_S},  {\rm id}_S).$$
By Lemmas \ref{liweia} and \ref{liweib}, one can show that  $\mathcal F$ and $\mathcal G$  are functors routinely.

Let $(T, +, \rhd)$ be a post Clifford semigroup denote $\mathcal{G}(T,+,\rhd)=((T,+), (T,\circ), L^{\rhd_T}, {\rm id}_T)$ and $\mathcal{FG}(T,+,\rhd)=(T,+,\rhd')$. Let $a,b\in T$. Then
$a\rhd'b=L^{\rhd_T}_{{\rm id}_T(a)}(b)=L^{\rhd_T}_a(b)=a\rhd b.$
This shows that $\rhd'$ is equal to $\rhd$. Thus  $\mathcal{FG}={\rm id}_{\mathbb{PCS}}$.

In the following, we shall prove that $\mathcal{G F}\cong{\rm id}_{\mathbb{BRRBCS}}$.
Let $(T, S, \phi, R)\in {\rm Obj}(\mathbb{BRRBCS})$.
Since $R$ is bijective, and is a  semigroup isomorphism from $(T, \circ_R)$ onto $(S, +)$ by Theorem \ref{h}. Obviously, ${\rm id}_T$ is a semigroup isomorphism
from $(T,+)$ onto $(T,+)$ and  $R\, {\rm id}_T=\, R{\rm id}_T$.  Let $a, b\in T$. Then
$L^{\rhd_{R}}_a(b)=a\rhd_{R} b=\phi_{R(a)}(b)$.
This gives $L^{\rhd_{R}}_a=\phi_{R(a)}$, and so ${\rm id}_TL^{\rhd_{R}}_a=\phi_{R(a)}{\rm id}_T$.
Thus in the category $\mathbb{BRRBCS}$,  $({\rm id}_T, R)$ is an isomorphism from $$(\mathcal{G}\mathcal{F})(T,S,\phi,R)=\mathcal{G}(T,+, \rhd_{R})=(T^{+},T^{\circ_{R}}, L^{\rhd_{R}},{\rm id}_T)$$
onto ${\rm id}_{\mathbb{BSRRBCS}}(T,S,\phi,R)=(T,S,\phi,R)$.
Denote $\Pi_{(T, S, \phi, R)}=({\rm id}_T, R)$ for each $(T, S, \phi, R)\in {\rm Obj}(\mathbb{BRRBCS})$ and let $$\Pi=\{\Pi_{(T,S,\phi,R)}\mid(T,S,\phi,R)\in  {\rm Obj}(\mathbb{BRRBCS})\}.$$
Assume that $$(T,S,\phi,R),(M,N,\varphi,B) \in {\rm Obj}(\mathbb{BRRBCS}), (\psi,\eta)\in {\rm Hom}((T,S,\phi, R),(M,N,\varphi,B)).$$
Then $\eta R=B\psi$ by Definition \ref{wang8}. Obviously,  $\psi {\rm id}_T=\psi= {\rm id}_M \psi$.
This implies that
$${\rm id}_{\mathbb{BRRBCS}}(\psi,\eta)\Pi_{(T,S,\phi,R)}=(\psi,\eta)({\rm id}_T,R)=(\psi {\rm id}_T,\eta R)
 $$$$=({\rm id}_M\psi, B \psi) )=({\rm id}_M, B)(\psi, \psi)
 =\Pi_{(M,N,\varphi,B)} (\mathcal{G}\mathcal{F})(\psi,\eta).$$
This shows that $\Pi$ is a natural isomorphism  from $\mathcal{G}\mathcal{F}$ to ${\rm id}_{\mathbb{BRRBCS}}$. So $\mathcal{G}\mathcal{F}\cong {\rm id}_{\mathbb{BRRBCS}}$. This together with $\mathcal{FG}={\rm id}_{\mathbb{PSC}}$ gives that $\mathbb{BRRBCS}$ is equivalent to  ${\mathbb{PSC}}$. The final statement follows from Lemmas \ref{liweia} and \ref{liweib}.
\end{proof}
The following proposition gives an answer of \cite[Question 2]{c1} for strong Rota--Baxter Clifford semigroups.
\begin{prop}\label{j} Let $(T, +, \circ_R)$ and $(T, +, \circ_B)$  be the dual weak left braces induced by strong relative Rota--Baxter Clifford semigroups $(T,S,\phi,R)$ and $(T,S,\varphi,B)$, respectively. Then $(T, +, \circ_R)\cong(T, +, \circ_B)$ if and only if there exists a semigroup isomorphism  $\psi_{T}: (T, {+})\rightarrow (T, {+})$ such that $\psi_{T}\phi_{R(a)}=\varphi_{B(\psi_{T}(a))}\psi_{T}$ for all $a\in T.$ In particular, if $(T,S,\phi,R)$ and $(T,S,\varphi,B)$ are  strong Rota--Baxter Clifford semigroups (that is, $T=S$ and $\phi_a(b)=a+b-a=\varphi_a(b)$ for all $a,b\in T$), then $(T, +, \circ_R)\cong(T, +, \circ_B)$ if and only if there exists a semigroup isomorphism  $\psi_{T}: (T, {+})\rightarrow (T, {+})$ such that $-\psi_T R(a)+B\psi_T(a)\in C(T).$
\end{prop}
\begin{proof} Let $\psi: (T, {+})\rightarrow (T, {+})$ be a semigroup isomorphism and let $a,b\in T$.  Then
\begin{equation*}
\begin{aligned}
&\psi(a\circ_{R}b)=\psi(a)\circ_{B}\psi(b)\\
&\Longleftrightarrow \psi(a+\phi_{R(a)}(b))=\psi(a)+\varphi_{B(\psi(a))}(\psi(b))\\
&\Longleftrightarrow \psi(a)+\psi\phi_{R(a)}(b)=\psi(a)+\varphi_{B(\psi(a))}(\psi(b))\\
&\Longleftrightarrow -\psi(a)+\psi(a)+\psi\phi_{R(a)}(b)=-\psi(a)+\psi(a)+\varphi_{B(\psi(a))}(\psi(b))\\
&\Longleftrightarrow \psi(a^{0}+\phi_{R(a)}(b))=\psi(a)^{0}+\varphi_{B(\psi(a))}(\psi(b))\\
&\Longleftrightarrow \psi(\phi_{R(a)}(b))=\varphi_{B(\psi(a))}(\psi(b))\\
&\Longleftrightarrow \psi\phi_{R(a)}(b)=\varphi_{B(\psi(a))}\psi(b).
\end{aligned}
\end{equation*}
This gives the first part. In the special case, we have
\begin{equation*}
\begin{aligned}
& \psi\phi_{R(a)}(b)=\varphi_{B(\psi(a))}\psi(b)\\
&\Longleftrightarrow  \psi(R(a)+b-R(a))=B(\psi(a))+\psi(b)-B(\psi(a))\\
&\Longleftrightarrow  \psi R(a)+\psi(b)-\psi R(a)=B(\psi(a))+\psi(b)-B(\psi(a))\\
&\Longleftrightarrow  \psi(b)-\psi R(a)=-\psi R(a)+B(\psi(a))+\psi(b)-B(\psi(a))\\
&\Longleftrightarrow  \psi(b)-\psi R(a)+B(\psi(a))=-\psi R(a)+B(\psi(a))+\psi(b)\\
&\Longleftrightarrow  -\psi R(a)+B\psi(a)\in C(T).
\end{aligned}
\end{equation*}
This shows the second part.
\end{proof}
\begin{prop}
Let $(T,S,\phi,R)$ and $(M,N,\varphi,B)$ be two strong  relative Rota--Baxter Clifford semigroups such that $R$ is a bijection. Then there is a one-to-one correspondence between ${\rm Hom}((T, +, \circ_R), (M, +, \circ_B))$ and ${\rm Hom}((T,S,\phi,R),(M,N,\varphi,B))$.
\end{prop}
\begin{proof} Let $\psi\in {\rm Hom}((T, +, \circ_R), (M, +, \circ_B))$.  Then $\psi\in {\rm Hom}((T,+), (M,+))$ obviously, and by the proof of the first part of Proposition \ref{j}, we have
\begin{equation}\label{wang7}
\psi\phi_{R(a)}=\varphi_{B(\psi(a))}\psi  \mbox{ for all } a\in T.
\end{equation}
Denote $\eta=B\psi R^{-1}$ and $x,y\in S$. Since $R$ is bijective, we can let $x=R(a)$ and $y=R(b)$ for some $a, b\in T$. Then
$x+y=R(a)+R(b)=R(a+\phi_{R(a)}(b)),$ and so
\begin{equation*}
\begin{aligned}
&\eta(x+y)=B\psi R^{-1}R(a+\phi_{R(a)}(b))=B\psi(a+\phi_{R(a)}(b))\\
&=B(\psi(a)+\psi\phi_{R(a)}(b))=B(\psi(a)+\varphi_{B(\psi(a))}\psi(b))\,\,\,\,\, (\mbox{by (\ref{wang7})})\\
&=B(\psi(a)\circ_{B}\psi(b))=B(\psi(a))+B(\psi(b))\,\,\,\,\, (\mbox{by  Lemma \ref{h}})\\
&=B\psi R^{-1}R(a)+B\psi R^{-1}R(b)=B\psi R^{-1}(x)+B\psi R^{-1}(y) =\eta(x)+\eta(y).
\end{aligned}
\end{equation*}
This shows  $\eta\in {\rm Hom}((S,+), (N,+))$. Since $\eta=B\psi R^{-1}$,  we have $\eta R=B\psi$. Moreover,
$$\psi\phi_{x}=\psi\phi_{R(a)}=\varphi_{B(\psi(a))}\psi=\varphi_{B\psi R^{-1}R(a)}\psi=\varphi_{\eta(R(a))}\psi=\varphi_{\eta(x)}\psi$$ by (\ref{wang7}).
This implies that $(\psi,\eta)\in {\rm Hom}((T,S,\phi,R),(M,N,\varphi,B))$. So we can define the following map:
$$\Theta: {\rm Hom}((T, +, \circ_R), (M, +, \circ_B))\rightarrow {\rm Hom}((T,S,\phi,R),(M,N,\varphi,B)),\,\, \psi\mapsto (\psi,B\psi R^{-1}).$$
Obviously, $\Theta$ is injective. Now let $(\psi, \eta)\in {\rm Hom}((T,S,\phi,R),(M,N,\varphi,B))$. Then by Definition \ref{wang8}, we have $\eta R=B\psi$. Since $R$ is bijective, we have $\psi\in {\rm Hom}((T, +, \circ_R), (M, +, \circ_B))$ and $\eta=B\psi R^{-1}$. This implies that $\Theta(\psi)=(\psi, B\psi R^{-1})=(\psi, \eta)$. Hence $\Theta$ is   surjective.
\end{proof}

In the end of this section, we consider the substructures and quotient structures of relative Rota--Baxter Clifford semigroups.
Let $(T,S,\varphi,R)$ be a relative Rota--Baxter Clifford semigroup and $M,N$ be Clifford subsemigroups of $T$ and $S$, respectively.  Assume that $R(M)\subseteq N$ and $\varphi_{n}(M)\subseteq M$ for all $n\in N$. Then we can define $$R|: M\rightarrow N,\, m\mapsto R(m) \mbox{ and } \varphi|: N\rightarrow \End M,\,\, n\mapsto \varphi_n|_M.$$ Then  $(M,N,\varphi|,R|)$ is also  a relative Rota--Baxter Clifford semigroup,  and is called {\em a relative Rota--Baxter Clifford subsemigroup of $(T,S,\varphi,R)$}.

\begin{defn}\label{4}
Let $(M,N,\varphi|,R|)$ be a relative Rota--Baxter Clifford subsemigroup of a relative Rota--Baxter Clifford  semigroup $(T,S,\varphi,R)$.   Then $(M,N,\varphi|,R|)$ is called {\em an ideal } of $(T,S,\varphi,R)$ if the following conditions hold:
\begin{itemize}
\item[\rm (I1)] $(M,+)$ and $(N, +)$ are normal subsemigroups of $T$ and $S$, respectively.
\item[\rm (I2)] $\varphi_{x}(M)\subseteq M$  for all $x\in S$.
\item[\rm (I3)] $\varphi_{n}(t)-t\in M$ for all $t\in T$ and $n\in N$.
\item[\rm (I4)] $\varphi_{n_{1}}(e)=\varphi_{n_{2}}(e)$ for all $e\in E(T,+)$ and $n_{1}, n_{2}\in N$ with $n^\circ_1=n^\circ_2$.
\end{itemize}
\end{defn}

\begin{theorem}\label{6}Let $(M,N,\varphi|,R|)$ be an ideal of a relative Rota--Baxter Clifford semigroup $(T,S,\varphi,R)$.
Define
$$\overline{\varphi}: S/N \rightarrow \End (T/M),\,\, a+N\mapsto \overline{\varphi}_{a+N},\,\,\,\,\,\, \overline{R}: T/M \rightarrow S/N,\,\,b+M \mapsto R(b)+N,$$ where $$\overline{\varphi}_{a+N}: T/M\rightarrow T/M,\,\, b+M\mapsto \varphi_{a}(b)+M.$$
Then $(T/M, S/N,\overline{\varphi}, \overline{R})$ is a relative Rota--Baxter Clifford semigroup.
\end{theorem}
\begin{proof} By (I1), (I2) and Lemma \ref{quotient},  the Clifford semigroups $(T/M, +)$ and $(S/N, +)$  in the theorem are meaningful.
Let $a,a_{1},a_{2}\in S,b,b_{1},b_{2}\in T$, if $b_{1}+M=b_{2}+M$, then $b_{1}^{0}=b_{2}^{0},-b_{1}+b_{2}\in M$. By (I2),
$$(\varphi_{a}(b_{1}))^{0}=\varphi_{a}(b_{1}^{0})=\varphi_{a}(b_{2}^{0})=\varphi_{a}(b_{2})^{0}, \,\, -\varphi_{a}(b_{1})+\varphi_{a}(b_{2})=\varphi_{a}(-b_{1}+b_{2})\in M.$$
This implies  $\varphi_{a}(b_{1})+M=\varphi_{a}(b_{2})+M$, and so $\overline{\varphi}_{a+N}$ is well-defined.
It is easy to  see that $\overline{\varphi}_{a+N}$ is a homomorphism. Now let $a_{1}+N=a_{2}+N$. Then $a_{1}^{0}=a_{2}^{0},-a_{1}+a_{2}\in N$,
By (I3),
$$b+\varphi_{-a_{1}+a_{2}}(-b)=b-\varphi_{-a_{1}+a_{2}}(b)=-(\varphi_{-a_{1}+a_{2}}(b)-b)\in M.$$
By (I2),
\begin{equation*}
\begin{aligned}
&\varphi_{a_{1}}(b)-\varphi_{a_{2}}(b)=\varphi_{a_{1}}(b)+\varphi_{a_{2}}(-b)=\varphi_{a_{1}}(b)+\varphi_{a_{2}^{0}+a_{2}}(-b)\\
&=\varphi_{a_{1}}(b)+\varphi_{a_{1}^{0}+a_{2}}(-b)=\varphi_{a_{1}}(b)+\varphi_{a_{1}-a_{1}+a_{2}}(-b) =\varphi_{a_{1}}(b+\varphi_{-a_{1}+a_{2}}(-b))\in M,
\end{aligned}
\end{equation*}
Since $a_{1}^{0}=a_{2}^{0}$, we have
$$(a_{1}^{0})^{0}=a_{1}^{0}=a_{1}^{0}+a_{2}^{0}=(-a_{1})^{0}+a_{2}^{0}=(-a_{1}+a_{2})^{0}.$$
Observe that $a_{1}^{0},-a_{1}+a_{2}\in N$, and $b^{0}\in E(S)$, it follows that
$\varphi_{a_{1}^{0}}(b^{0})=\varphi_{-a_{1}+a_{2}}(b^{0}).$
This implies that
$$
\varphi_{a_{1}}(b)^{0}=\varphi_{a_{1}}(b^{0})=\varphi_{a_{1}}\varphi_{a_{1}^{0}}(b^{0})=\varphi_{a_{1}}\varphi_{-a_{1}+a_{2}}(b^{0})
=\varphi_{a_{1}-a_{1}+a_{2}}(b^{0})$$$$=\varphi_{a_{1}^{0}+a_{2}}(b^{0})=\varphi_{a_{2}^{0}+a_{2}}(b^{0})
=\varphi_{a_{2}}(b^{0})=\varphi_{a_{2}}(b)^{0}.
$$
Thus
$$\overline{\varphi}_{a_{1}+N}(b+M)=\varphi_{a_{1}}(b)+M
=\varphi_{a_{2}}(b)+M=\overline{\varphi}_{a_{2}+N}(b+M).$$
It is routine to check that $\overline{\varphi}$ is a homomorphism.

Let $b_{1}+M=b_{2}+M$. Then $b_{1}^{0}=b_{2}^{0},-b_{1}+b_{2}\in M$, and so $\varphi_{-R(b_{1})}(-b_{1}+b_{2})\in M$  by (I2). This implies that
\begin{equation*}
\begin{aligned}
&-R(b_{1})+R(b_{2})=R(\varphi_{-R(b_{1})}(-b_{1}))+R(b_{2})=R(\varphi_{-R(b_{1})}(-b_{1})+\varphi_{R(\varphi_{-R(b_{1})}(-b_{1}))}(b_{2}))\\
&=R(\varphi_{-R(b_{1})}(-b_{1})+\varphi_{-R(b_{1})}(b_{2}))=R(\varphi_{-R(b_{1})}(-b_{1}+b_{2}))\in N.\\
\end{aligned}
\end{equation*}
Observe that $R(b_{1})^{0}=R(b_{1}^{0})=R(b_{2}^{0})=R(b_{2})^{0}$, it follows that $R(b_{1})+N=R(b_{2})+N$. So $\overline{R}$ is well-defined.
Let $b,b_{1},b_{2}\in T$. Then
$$\overline{R}(b+M)^{0}=\overline{R}(b+M)-\overline{R}(b+M)=(R(b)+N)-(R(b)+N)
$$$$=(R(b)-R(b))+N=R(b)^{0}+N=R(b^{0})+N=\overline{R}(b^{0}+M),$$
$$\overline{\varphi}_{\overline{R}(b+M)^{0}}(b+M)
=\overline{\varphi}_{R(b^{0})+N}(b+M)
=\varphi_{R(b)^{0}}(b)+M=b+M,$$
$$
\overline{R}(b_{1}+M+\overline{\varphi}_{\overline{R}(b_{1}+M)}(b_{2}+M))
=\overline{R}(b_{1}+M+\overline{\varphi}_{R(b_{1})+N}(b_{2}+M))$$$$
=\overline{R}(b_{1}+M+\varphi_{R(b_{1})}(b_{2})+M)
=\overline{R}((b_{1}+\varphi_{R(b_{1})}(b_{2}))+M)=R(b_{1}+\varphi_{R(b_{1})}(b_{2}))+N$$$$
=(R(b_{1})+R(b_{2}))+N=R(b_{1})+N+R(b_{2})+N
=\overline{R}(b_{1}+M)+\overline{R}(b_{2}+M),
$$
$$
b_{1}+M+\overline{\varphi}_{\overline{R}(b_{1}+M)}((b_{2}+M)^{0})
=b_{1}+M+\overline{\varphi}_{R(b_{1})+N}(b_{2}^{0}+M)
=b_{1}+M+\varphi_{R(b_{1})}(b_{2}^{0})+M$$$$
=(b_{1}+\varphi_{R(b_{1})}(b_{2}^{0}))+M=(b_{2}^{0}+\varphi_{R(b_{2}^{0})}(b_{1}))+M
=b_{2}^{0}+M+\varphi_{R(b_{2}^{0})}(b_{1})+M$$$$
=(b_{2}+M)^{0}+\overline{\varphi}_{R(b_{2}^{0})+N}(b_{1}+M)
=(b_{2}+M)^{0}+\overline{\varphi}_{\overline{R}(b_{2}^{0}+M)}(b_{1}+M)$$$$
=(b_{2}+M)^{0}+\overline{\varphi}_{\overline{R}((b_{2}+M)^{0})}(b_{1}+M).
$$
This shows that $(T/M,S/N,\overline{\varphi},\overline{R})$ is a relative Rota-Baxter Clifford semigroup.
\end{proof}

\begin{prop}Let $(M,N,\varphi|,R|)$ be an ideal of  a relative Rota--Baxter Clifford semigroup $(T,S,\varphi,R)$. Then $(M,+,\circ_{R|})$ is an ideal of $(T,+,\circ_{R})$ and  the descendent dual weak left brace $(T/M, +, \circ_{{R|}})$  of  $(T/M, S/N,\overline{\varphi},\overline{R})$ is exactly the quotient dual weak left brace $(T,+,\circ_{R})/(M,+,\circ_{R|})$.
\end{prop}
\begin{proof}
By hypothesis, $(M,+)$ is a normal subsemigroup of $(T,+)$, and so $E(T,+)\subseteq (M,+)$. Let $a\in T,b\in M$. Then $\varphi_{R(a)}(b)\in M$  by(b), and so
$$\lambda_{a}^{M}(b)=-a+a\circ_{R}b=-a+a+\varphi_{R(a)}(b)=a^{0}+\varphi_{R(a)}(b)\in M,$$
This implies that $\lambda_{a}^{M}(M)\subseteq M$.

Let $a,b\in M$. Then $-a\in M$ by (I1), and so $a^{-1}=\varphi_{-R(a)}(-a)\in M$ by (I2).
Let $a\in T,b\in M$. Then
$$
a\circ_{R}b\circ_{R}a^{(-1)}=(a+\varphi_{R(a)}(b))\circ_{R}a^{(-1)}
=a+\varphi_{R(a)}(b)+\varphi_{(a+\varphi_{R(a)}(b))}\varphi_{-R(a)}(-a)
$$$$=a+\varphi_{R(a)}(b)+\varphi_{R(a)+R(b)}\varphi_{-R(a)}(-a)
=a+\varphi_{R(a)}(b)-a+a-\varphi_{R(a)+R(b)-R(a)}(a).
$$
By (I2),  $\varphi_{R(a)}(b)\in M$. This together with the fact that $(M,+)$ is a normal subsemigroup gives that $a+\varphi_{R(a)}(b)-a\in M$. Since $R(M)\subseteq N$ and $(N,+)$ is a normal subsemigroup of $S$, it follows that $R(b)\in N$ and $R(a)+R(b)-R(a)\in N$. In view of (I4), we have $a-\varphi_{R(a)+R(b)-R(a)}(a)\in M$. Thus
$$a\circ_{R}b\circ_{R}a^{(-1)}=a+\varphi_{R(a)}(b)-a+a-\varphi_{R(a)+R(b)-R(a)}(a)\in M.$$
Thus $(M,\circ_{R})$ is a normal subsemigroup of $(T,\circ_{R})$. So $(M,+,\circ_{R|})$ is an ideal of $(T,+,\circ_{R})$. The remaining part is obvious.
\end{proof}

\section{Post Clifford semigroups and braided Clifford semigroups}
In this section, we introduce  braided Clifford semigroups and prove that the categories of strong post Clifford semigroups and braided Clifford semigroups are isomorphic.
\begin{defn}
Let $(T,\circ)$ be a Clifford semigroup and $$\sigma:T\times T\rightarrow T\times T, (a,b)\mapsto (a\rightharpoonup b,a\leftharpoonup b)$$ be a map. Then $(T,\circ,\sigma)$ is called a {\em braided Clifford semigroup} if the following conditions are satisfied: For all $a,b,a_{1},a_{2},b_{1},b_{2}\in T$,
\begin{itemize}
\item[(B1)] $a^{0}\rightharpoonup a=a$,
\item[(B2)] $a_{1}\rightharpoonup (a_{2}\rightharpoonup b)=(a_{1}\circ a_{2})\rightharpoonup b$,
\item[(B3)] $(a_{1}\circ a_{2})\leftharpoonup b=(a_{1}\leftharpoonup(a_{2}\rightharpoonup b))\circ (a_{2}\leftharpoonup b)$,
\item[(B4)] $a\leftharpoonup a^{0}=a$,
\item[(B5)] $(a\leftharpoonup b_{1})\leftharpoonup b_{2}=a\leftharpoonup (b_{1}\circ b_{2})$,
\item[(B6)] $a\rightharpoonup (b_{1}\circ b_{2})=(a\rightharpoonup b_{1})\circ ((a\leftharpoonup b_{1})\rightharpoonup b_{2})$,
\item[(B7)] $(a\rightharpoonup b)\circ(a\leftharpoonup b)=a\circ b$,
\item[(B8)] $(a\rightharpoonup b)^{0}=(a\leftharpoonup b)^{0}$,
\item[(B9)] $a\circ(a^{0}\rightharpoonup b)=a\circ b$.
\end{itemize}
\end{defn}
\begin{exam}Let $(T,\circ)$ be a commutative and idempotent semigroup. Define $$\sigma:T\times T\rightarrow T\times T, (a,b)\mapsto (a\rightharpoonup b,a\leftharpoonup b)=(a\circ b, a\circ b).$$ Then $a^0=a$ for all $a\in S$.  It  is easy to see that $(T,\circ,\sigma)$ is a braided Clifford semigroup.
\end{exam}
\begin{exam}
Let $(G, \circ)$ be a group and $$\sigma: T\times T\rightarrow T\times T, (a,b)\mapsto (a\rightharpoonup b,a\leftharpoonup b)$$ be a map. Then $(G, \circ)$ is a Clifford semigroup and $a^0$ is the identity in $G$ for all $a\in G$ obviously. Recall from Definitions 3.8 and 3.11 in \cite{Bai} (also see \cite{Takeuchi}) that $(G,\circ,\sigma)$ is called a {\em braided group} if (B1)--(B7) hold for all $a,b,a_{1},a_{2},b_{1},b_{2}\in G$. Moreover, in this case, (B8) is satisfied trivially, and (B9) follows from (B1). Thus braided groups are necessarily braided Clifford semigroups.
\end{exam}

\begin{lemma}\label{x3}
Let $(T,\circ,\sigma)$ be a braided Clifford semigroup with $$\sigma:T\times T\rightarrow T\times T, (a,b)\mapsto (a\rightharpoonup b,a\leftharpoonup b).$$ Assume that $a,b,c\in T$. Then
\begin{itemize}
\item[\rm(1)]  $(a\rightharpoonup b)^{0}=(a\leftharpoonup b)^{0}=a^{0}\circ b^{0}$,
\item[\rm(2)] $a^{0}\circ (a\rightharpoonup b)=a\rightharpoonup b$,
\item[\rm(3)] $(a\rightharpoonup b)^{-1}\rightharpoonup (a\rightharpoonup c)=(a\leftharpoonup b)\rightharpoonup (b^{-1}\rightharpoonup c)$,
\item[\rm(4)] $a\rightharpoonup a^{0}=a^0$.
\end{itemize}
\end{lemma}
\begin{proof}
 (1) By (B7) and (B8), we have
\begin{equation*}
\begin{aligned}
a^{0}\circ b^{0}&=(a\circ b)^{0}=((a\rightharpoonup b)\circ(a\leftharpoonup b))^{0}=(a\rightharpoonup b)^{0}\circ (a\leftharpoonup b)^{0}\\
&=(a\rightharpoonup b)^{0}\circ (a\rightharpoonup b)^{0}=(a\rightharpoonup b)^{0}=(a\leftharpoonup b)^{0}.
\end{aligned}
\end{equation*}
(2) By item (1) above, we have
\begin{equation*}
\begin{aligned}
a^{0}\circ (a\rightharpoonup b)&=a^{0}\circ (a\rightharpoonup b)^{0} \circ (a\rightharpoonup b)
=a^{0}\circ a^{0}\circ b^{0}\circ (a\rightharpoonup b)\\
&=a^{0}\circ b^{0}\circ (a\rightharpoonup b)=(a\rightharpoonup b)^{0} \circ (a\rightharpoonup b)=a\rightharpoonup b.
\end{aligned}
\end{equation*}
(3) We observe that
\begin{equation*}
\begin{aligned}
&(a\rightharpoonup b)\rightharpoonup((a\rightharpoonup b)^{-1}\rightharpoonup (a\rightharpoonup c))
=((a\rightharpoonup b)\circ (a\rightharpoonup b)^{-1})\rightharpoonup (a\rightharpoonup c) \quad\quad(\mbox{by (B2)})\\
&=(a\rightharpoonup b)^{0}\rightharpoonup (a\rightharpoonup c)
=(a^{0}\circ b^{0})\rightharpoonup (a\rightharpoonup c) \quad(\mbox{by  Lemma \ref{x3} (1)})\\
&=(a^{0}\circ b^{0}\circ a)\rightharpoonup c
=(a\circ b^{0})\rightharpoonup c, \quad(\mbox{by (B2)})\\
\end{aligned}
\end{equation*}
\begin{equation*}
\begin{aligned}
&(a\rightharpoonup b)\rightharpoonup((a\leftharpoonup b)\rightharpoonup (b^{-1}\rightharpoonup c))
=((a\rightharpoonup b)\circ (a\leftharpoonup b))\rightharpoonup(b^{-1}\rightharpoonup c)  \quad\quad(\mbox{by  (B2)})\\
&=(a\circ b)\rightharpoonup (b^{-1}\rightharpoonup c)
=(a\circ b\circ b^{-1})\rightharpoonup c
=(a\circ b^{0})\rightharpoonup c.\quad\quad(\mbox{by (B7) and (B2)})\\
\end{aligned}
\end{equation*}
This implies that
$$(a\rightharpoonup b)^{-1}\rightharpoonup ((a\rightharpoonup b)\rightharpoonup ((a\rightharpoonup b)^{-1}\rightharpoonup (a\rightharpoonup c)))$$$$=(a\rightharpoonup b)^{-1}\rightharpoonup ((a\rightharpoonup b)\rightharpoonup((a\leftharpoonup b)\rightharpoonup(b^{-1}\rightharpoonup c))).$$
By (B2), we obtain
\begin{equation*}
\begin{aligned}
&(a\rightharpoonup b)^{-1}\rightharpoonup (a\rightharpoonup c)=((a\rightharpoonup b)^{-1}\circ (a\rightharpoonup b)\circ (a\rightharpoonup b)^{-1})\rightharpoonup (a\rightharpoonup c)\\
&=((a\rightharpoonup b)^{-1}\circ (a\rightharpoonup b)\circ (a\leftharpoonup b))\rightharpoonup (b^{-1}\rightharpoonup c)\\
&=((a\rightharpoonup b)^{0}\circ (a\leftharpoonup b))\rightharpoonup(b^{-1}\rightharpoonup c)
=(a\leftharpoonup b)\rightharpoonup (b^{-1}\rightharpoonup c). \quad\quad(\mbox{by  (B8)})
\end{aligned}
\end{equation*}
(4) Observe that
\begin{equation*}
\begin{aligned}
&(a\rightharpoonup a^{0})\circ (a\rightharpoonup a^{0})=(a\rightharpoonup a^{0})\circ ((a\leftharpoonup a^{0})\rightharpoonup a^{0}) \quad\quad(\mbox{by  (B4)})\\
&=a\rightharpoonup (a^{0}\circ a^{0})
=a\rightharpoonup a^{0}\in E(T,\circ). \quad\quad(\mbox{by (B6)})
\end{aligned}
\end{equation*}
This implies that $a\rightharpoonup a^{0}=(a\rightharpoonup a^{0})^{0}=a^{0}\circ a^{00}=a^{0}$ by item (1) in the present  lemma.
\end{proof}

\begin{lemma}\label{x8}
Let $(T,+,\rhd)$ be a post Clifford semigroup with sub-adjacent Clifford semigroup $(T,\circ)$. Define a map $$R_{T}:T\times T\rightarrow T\times T, (a,b)\mapsto (a\rightharpoonup b,a\leftharpoonup b),$$where $a\rightharpoonup b=a^{0}+(a\rhd b)=a^{0}+L_{a}^{\rhd}(b)$ and $a\leftharpoonup b=(a\rhd b)^{-1}\circ a\circ b$ for all $a,b\in T$. Then $(T,\circ,R_{T})$ forms a braided Clifford semigroup and $R_T$ is a solution of the Yang-Baxter equation on the set $T$. If $\psi:(T,+,\rhd)\rightarrow (S,+,\rhd)$ is a Post-Clifford semigroup homomorphism, then $\psi$ is also a braided Clifford semigroup homomorphism.
\end{lemma}
\begin{proof} Let $a,b,c\in T$. Then we have
 \begin{equation}
(a\rightharpoonup b)^{0}=(a^{0}+(a\rhd b))^{0}=a^{00}+(a\rhd b)^{0}=a^{0}+(a\rhd b)^{0}=a^{0}+(a\rhd b^{0}),
\end{equation}
\begin{equation}\label{x4}
a^{0}+(a\rhd b^{0})=-a+a+(a\rhd b^{0})=-a+a\circ b^{0}=-a+a+b^{0}=a^{0}+b^{0},
\end{equation}
\begin{equation}\label{x5}
\begin{aligned}
&(a\rhd b)^{0}\circ a\circ b=(a\rhd b)^{0}\circ a^{0}\circ a\circ b=a^{0}\circ (a\rhd b)^{0}\circ a\circ b\\
&=(a^{0}+(a\rhd b)^{0})\circ a\circ b=(a+(a\rhd b))^{0}\circ a\circ b =(a\circ b)^{0}\circ a\circ b=a\circ b.
\end{aligned}
\end{equation}
Thus
\begin{equation*}
\begin{aligned}
(a\leftharpoonup b)^{0}&=((a\rhd b)^{-1}\circ a\circ b)^{0}=(a\rhd b)^{0}\circ a^{0}\circ b^{0}=(a\rhd b)^{00}\circ a^{0}\circ b^{0}\\
&=((a\rhd b)^{0}\circ a\circ b)^{0}=(a\circ b)^{0}=a^{0}\circ b^{0}=a^{0}+b^{0}=(a\rightharpoonup b)^{0}.
\end{aligned}
\end{equation*}
This shows that (B8) holds.  By (\ref{x3}), we have
$$(a\rightharpoonup b)\circ (a\leftharpoonup b)=(a^{0}+(a\rhd b))\circ ((a\rhd b)^{-1}\circ a\circ b)$$$$=a^{0}\circ (a\rhd b)\circ (a\rhd b)^{-1}\circ a\circ b
=(a\rhd b)^{0}\circ a\circ b=a\circ b.$$ This shows that (B7) holds. (B9) follows from the following equation $$a\circ (a^{0}\rightharpoonup b)=a\circ (a^{0}+(a^{0}\rhd b))=a\circ (a^{0}\circ b)=a\circ b,$$ and (B1)
follows from the fact that $a^{0}\rightharpoonup a=a^{00}+(a^{0}\rhd a)=a^{0}+a=a.$ Moreover,
\begin{eqnarray*}
&&a_{1}\rightharpoonup (a_{2}\rightharpoonup b)=a_{1}^{0}+(a_{1}\rhd (a_{2}\rightharpoonup b))=a_{1}^{0}+(a_{1}\rhd (a_{2}^{0}+(a_{2}\rhd b)))\\
&=&a_{1}^{0}+(a_{1}\rhd a_{2}^{0})+(a_{1}\rhd(a_{2}\rhd b))\\
&=&a_{1}^{0}+a_{2}^{0}+(a_{1}\circ a_{2})\rhd b  \quad\quad (\mbox{by  (\ref{x4}) and  Theorem \ref{q3} (2)})\\
&=&a_{1}^{0}\circ a_{2}^{0}+(a_{1}\circ a_{2})\rhd b=(a_{1}\circ a_{2})^{0}+(a_{1}\circ a_{2})\rhd b =(a_{1}\circ a_{2})\rightharpoonup b.
\end{eqnarray*}
Thus (B2) holds, and (B3) follows from the fact that
\begin{equation*}
\begin{aligned}
&(a_{1}\leftharpoonup (a_{2}\rightharpoonup b))\circ (a_{2}\leftharpoonup b)
=(a_{1}\rhd (a_{2}\rightharpoonup b))^{-1}\circ a_{1}\circ (a_{2}\rightharpoonup b)\circ (a_{2}\leftharpoonup b)\\
&=(a_{1}\rhd (a_{2}\rightharpoonup b))^{-1}\circ a_{1}^{0} \circ a_{1}\circ a_{2}\circ b  \quad \quad (\mbox{by  (B7)})\\
&=(a_{1}^{0}\circ (a_{1}\rhd (a_{2}\rightharpoonup b)))^{-1}\circ a_{1}\circ a_{2}\circ b \\
&=(a_{1}^{0}+ (a_{1}\rhd (a_{2}\rightharpoonup b)))^{-1}\circ a_{1}\circ a_{2}\circ b \\
&=(a_{1}\rightharpoonup (a_{2}\rightharpoonup b))^{-1}\circ a_{1}\circ a_{2}\circ b \\
&=((a_{1}\circ a_{2})\rightharpoonup b)^{-1}\circ a_{1}\circ a_{2}\circ b  \quad \quad (\mbox{by  (B2)})\\
&=((a_{1}\circ a_{2})^{0}+((a_{1}\circ a_{2})\rhd b))^{-1}\circ a_{1}\circ a_{2}\circ b \\
&=((a_{1}\circ a_{2})^{0}\circ ((a_{1}\circ a_{2})\rhd b))^{-1}\circ a_{1}\circ a_{2}\circ b \\
&=((a_{1}\circ a_{2})\rhd b)^{-1}\circ (a_{1}\circ a_{2})^{0}\circ a_{1}\circ a_{2}\circ b \\
&=((a_{1}\circ a_{2})\rhd b)^{-1}\circ a_{1}\circ a_{2}\circ b=(a_{1}\circ a_{2})\leftharpoonup b.
\end{aligned}
\end{equation*}
By Lemma \ref{q2}, we have $a\leftharpoonup a^{0}=(a\rhd a^{0})^{-1}\circ a \circ a^{0}=a^{0}\circ a\circ a^{0}=a,$ which gives (B4).
Observe that
\begin{equation}\label{x6}
\begin{aligned}
&(a\rhd b_{1})\circ (((a\rhd b_{1})^{-1}\circ a\circ b_{1})\rhd b_{2})\\
&=(a\rhd b_{1})+(a\rhd b_{1})\rhd (((a\rhd b_{1})^{-1}\circ a\circ b_{1})\rhd b_{2})\\
&=(a\rhd b_{1})+(((a\rhd b_{1})\circ (a\rhd b_{1})^{-1}\circ a\circ b_{1})\rhd b_{2}) \quad (\mbox{by Theorem \ref{q3} (2)}) \\
&=(a\rhd b_{1})+(((a\rhd b_{1})^{0}\circ a\circ b_{1})\rhd b_{2})\\
&=(a\rhd b_{1})+((a\circ b_{1})\rhd b_{2}) \quad\quad(\mbox{by (\ref{x3})})\\
&=(a\rhd b_{1})+(a\rhd( b_{1}\rhd b_{2})) \quad\quad(\mbox{by Theorem \ref{q3} (2)})\\
&=a\rhd (b_{1}+(b_{1}\rhd b_{2}))=a\rhd(b_{1}\circ b_{2}),
\end{aligned}
\end{equation}
it follows that
\begin{equation*}
\begin{aligned}
&(a\leftharpoonup b_{1})\leftharpoonup b_{2}=((a\rhd b_{1})^{-1}\circ a\circ b_{1})\leftharpoonup b_{2}\\
&=(((a\rhd b_{1})^{-1}\circ a\circ b_{1})\rhd b_{2})^{-1}\circ (a\rhd b_{1})^{-1}\circ a\circ b_{1}\circ b_{2}\\
&=((a\rhd b_{1})\circ (((a\rhd b_{1})^{-1}\circ a\circ b_{1})\rhd b_{2}))^{-1}\circ a\circ b_{1}\circ b_{2},\\
&=(a\rhd (b_{1}\circ b_{2}))^{-1}\circ a\circ b_{1}\circ b_{2}=a\leftharpoonup (b_{1}\circ b_{2}).
\end{aligned}
\end{equation*}
Thus (B5) is true.  (B6) follows from the fact that
\begin{equation*}
\begin{aligned}
&(a\rightharpoonup b_{1})\circ ((a\leftharpoonup b_{1})\rightharpoonup b_{2})
=(a\rightharpoonup b_{1})\circ ((a\leftharpoonup b_{1})^{0}+(a\leftharpoonup b_{1})\rhd b_{2})\\
&=(a\rightharpoonup b_{1})\circ (a\leftharpoonup b_{1})^{0}\circ (a\leftharpoonup b_{1})\rhd b_{2})\\
&=(a\rightharpoonup b_{1})\circ (a\leftharpoonup b_{1})\rhd b_{2})  \quad\quad (\mbox{by  (B8)})\\
&=(a^{0}+(a\rhd b_{1}))\circ(((a\rhd b_{1})^{-1}\circ a\circ b_{1})\rhd b_{2})\\
&=a^{0}\circ (a\rhd b_{1})\circ(((a\rhd b_{1})^{-1}\circ a\circ b_{1})\rhd b_{2})\\
&=a^{0}\circ (a\rhd( b_{1}\circ b_{2}) \quad\quad(\mbox{by   (\ref{x6})})\\
&=a^{0}+a\rhd(b_{1}\circ b_{2})=a\rightharpoonup (b_{1}\circ b_{2}).
\end{aligned}
\end{equation*}
By the above statements, $(T,\circ, R_{T})$ forms a braided Clifford semigroup, and $R_T$ is a solution of the Yang-Baxter equation on the set $T$
by Theorem \ref{postjie}.

If $\psi:(T,+,\rhd)\rightarrow (S,+,\rhd)$ is a post Clifford semigroup homomorphism, then by Theorem \ref{q3} (2), we have $\psi(a\circ b)=\psi(a)\circ \psi(b)$.  Moreover,
\begin{equation*}
\begin{aligned}
&(\psi\times\psi)R_{T}(a,b)=(\psi\times\psi)(a^{0}+a\rhd b,(a\rhd b)^{-1}\circ a\circ b)\\
&=(\psi(a^{0}+a\rhd b),\psi((a\rhd b)^{-1}\circ a\circ b))\\
&=(\psi(a)^{0}+\psi(a)\rhd \psi(b),(\psi(a)\rhd \psi(b))^{-1}\circ \psi(a)\circ \psi(b))\\
&=R_{S}(\psi(a),\psi(b))=R_{S}(\psi\times\psi)(a,b).
\end{aligned}
\end{equation*}
Thus $\psi$ is also a braided Clifford semigroup homomorphism.
\end{proof}

\begin{lemma}\label{x9}
Let $(T,\circ,R_{T})$ be a braided Clifford semigroup with $$R_{T}:T\times T\rightarrow T\times T,(a,b)\mapsto (a\rightharpoonup b,a\leftharpoonup b).$$ Define two binary operations $``+"$ and $``\rhd"$ as follows:
$$a+b=a\circ (a^{-1}\rightharpoonup b) \quad and \quad a\rhd b={a}\rightharpoonup{b}$$
for all $a,b\in T$. Then $(T,+,\rhd)$ forms a strong post Clifford semigroup and $-a=a\rightharpoonup a^{-1}$ for all $a\in S$. If $\psi:(T,\circ,R_{T})\rightarrow (S,\circ,R_{S})$ is a braided Clifford semigroup homomorphism, then $\psi$ is a post Clifford semigroup homomorphism.
\end{lemma}
\begin{proof}
Let $a,b,c\in T$. Then
\begin{equation*}
\begin{aligned}
&(a+b)+c=a\circ (a^{-1}\rightharpoonup b)+c=a\circ (a^{-1}\rightharpoonup b)\circ((a\circ (a^{-1}\rightharpoonup b))^{-1}\rightharpoonup c)\\
&=a\circ (a^{-1}\rightharpoonup b)\circ(((a^{-1}\rightharpoonup b)^{-1}\circ a^{-1})\rightharpoonup c)\\
&=a\circ (a^{-1}\rightharpoonup b)\circ((a^{-1}\rightharpoonup b)^{-1}\rightharpoonup (a^{-1}\rightharpoonup c)) \quad (\mbox{by  (B2)})\\
&=a\circ (a^{-1}\rightharpoonup b)\circ((a^{-1}\leftharpoonup b)\rightharpoonup (b^{-1}\rightharpoonup c))  \quad ((\mbox{by Lemma \ref{x3} (3)})\\
&=a\circ (a^{-1}\rightharpoonup (b\circ (b^{-1}\rightharpoonup c))) \quad (\mbox{by  (B6)})\\
&=a\circ (a^{-1}\rightharpoonup (b+c))=a+(b+c).
\end{aligned}
\end{equation*}
This shows that $(T, +)$ is a semigroup. Observe that
\begin{equation}\label{xx}
\begin{aligned}
a+(a\rightharpoonup a^{-1})&=a\circ (a^{-1}\rightharpoonup (a\rightharpoonup a^{-1}))=a\circ ((a^{-1}\circ a)\rightharpoonup a^{-1}) \quad (\mbox{by (B2)})\\
&=a\circ (a^{0}\rightharpoonup a^{-1})=a\circ a^{-1}=a^{0}. \quad (\mbox{by (B9)})\\
\end{aligned}
\end{equation}
By Lemma \ref{x3}(1), we have $(a\rightharpoonup a^{-1})^{0}=(a\leftharpoonup a^{-1})^{0}=a^{0}\circ(a^{-1})^{0}=a^{0}$, and so
\begin{equation}\label{x7}
(a\rightharpoonup a^{-1})^{-1}=(a\leftharpoonup a^{-1}).
\end{equation}
Since  $a+(a\rightharpoonup a^{-1})+a=a^{0}+a=a^{0}\circ ((a^{0})^{-1}\rightharpoonup a)=a^{0}\circ(a^{0}\rightharpoonup a)=a^{0}\circ a=a$ and
\begin{equation*}
\begin{aligned}
&(a\rightharpoonup a^{-1})+a+(a\rightharpoonup a^{-1})=(a\rightharpoonup a^{-1})+a^{0}=(a\rightharpoonup a^{-1})\circ ((a\rightharpoonup a^{-1})^{-1}\rightharpoonup a^{0})\\
&=(a\rightharpoonup a^{-1})\circ ((a\leftharpoonup a^{-1})\rightharpoonup a^{0})=a\rightharpoonup(a^{-1}\circ a^{0}) =a\rightharpoonup a^{-1}, \quad (\mbox{by (B6)})
\end{aligned}
\end{equation*}
it follows that $(T,+)$ is regular. If $a\in E(T,+)$, then $a=a+a=a\circ(a^{-1}\rightharpoonup a)$, and so
$$a^{0}=a^{-1}\circ a=a^{-1}\circ a\circ(a^{-1}\rightharpoonup a)=a^{0}\circ(a^{-1}\rightharpoonup a)=(a^{-1})^{0}\circ(a^{-1}\rightharpoonup a)=a^{-1}\rightharpoonup a.$$
By Lemma \ref{x3} (2), (B1), (B2) and Lemma \ref{x3} (4),
$$a=a^{0}\rightharpoonup a=(a\circ a^{-1})\rightharpoonup a=a\rightharpoonup(a^{-1}\rightharpoonup a)=a\rightharpoonup a^{0}= a^{0}.$$ 
Now let $a,b\in E(T,+)$. Then $a=a^{0}$ and $b=b^{0}$. This implies that
$$a+b=a^{0}+b^{0}=a^{0}\circ(a^{0}\rightharpoonup b^{0})=a^{0}\circ (a^{00}\rightharpoonup b^{0})=a^{0}\circ b^{0}.$$
Dually, $b+a=b^{0}\circ a^{0}$. This yields that $a+b=a^{0}\circ b^{0}=b^{0}\circ a^{0}=b+a$. Thus $(T,+)$ is an inverse semigroup. Since
\begin{equation*}
\begin{aligned}
&(a\rightharpoonup a^{-1})+a=(a\rightharpoonup a^{-1})\circ ((a\rightharpoonup a^{-1})^{-1}\rightharpoonup a)\\
&=(a\rightharpoonup a^{-1})\circ ((a\leftharpoonup a^{-1})\rightharpoonup a) \quad (\mbox{by  (\ref{x7})})  \\
&=a\rightharpoonup (a^{-1}\circ a)=a\rightharpoonup a^{0}=a^{0}, \quad (\mbox{by (B6) and  Lemma \ref{x3} (4)})
\end{aligned}
\end{equation*}
we have $a+(a\rightharpoonup a^{-1})=(a\rightharpoonup a^{-1})+a$ by (\ref{xx}). So $(T,+)$ is a Clifford semigroup.
On the other hand,
\begin{equation*}
\begin{aligned}
&a\rhd(b+c)=a\rightharpoonup(b\circ (b^{-1}\rightharpoonup c))
=(a\rightharpoonup b)\circ ((a\leftharpoonup b)\rightharpoonup(b^{-1}\rightharpoonup c))  \quad (\mbox{by (B6)})\\
&=(a\rightharpoonup b)\circ ((a\rightharpoonup b)^{-1}\rightharpoonup (a\rightharpoonup c)) \quad (\mbox{by Lemma \ref{x3} (3)})\\
&=(a\rightharpoonup b)+(a\rightharpoonup c)=(a\rhd b)+(a\rhd c).
\end{aligned}
\end{equation*}
This gives (P1). (P2) and (P3) follow  from the equations
\begin{equation*}
\begin{aligned}
&(a+(a\rhd b))\rhd c=(a+(a\rightharpoonup b))\rightharpoonup c=(a\circ (a^{-1}\rightharpoonup (a\rightharpoonup b)))\rightharpoonup c\\
&=(a\circ((a^{-1}\circ a)\rightharpoonup b))\rightharpoonup c=(a\circ(a^{0}\rightharpoonup b))\rightharpoonup c \quad (\mbox{by  (B2)})\\
&=(a\circ b)\rightharpoonup c=a\rightharpoonup(b\rightharpoonup c)=a\rhd(b\rhd c)  \quad (\mbox{by  (B9)  and   (B2)})\\
\end{aligned}
\end{equation*}
and
\begin{equation*}
\begin{aligned}
&a+(a\rhd b^{0})=a+(a\rightharpoonup b^{0})=a\circ(a^{-1}\rightharpoonup (a\rightharpoonup b^{0}))
=a\circ((a^{-1}\circ a)\rightharpoonup b^{0}) \quad (\mbox{by (B2)})\\
&=a\circ(a^{0}\rightharpoonup b^{0})=a\circ b^{0}=b^{0}\circ a=b^{0}\circ(b^{00}\rightharpoonup a) \quad (\mbox{by  (B9)})\\
&=b^{0}\circ(b^{0}\rightharpoonup a)=b^{0}\circ(b^{0}\rightharpoonup (b^{0}\rightharpoonup a)) \quad (\mbox{by  (B2)})\\
&=b^{0}+(b^{0}\rightharpoonup a)=b^{0}+(b^{0}\rhd a).
\end{aligned}
\end{equation*}
To see (P4), we first oberve that $H_{a}=\{x\in T\mid x^{0}=a^{0}\}$ and $H_{a}=H_{a^{-1}}$. Let $L_{a}^{\rhd}:T\rightarrow T,b\mapsto a\rhd b$ and $b\in H_{a}, c\in T$. Then $b^{0}=a^{0}$. By Lemma \ref{x3} (1), we have
$$(L_{a}^{\rhd}(b))^{0}=(a\rhd b)^{0}=(a\rightharpoonup b)^{0}=a^{0}\circ b^{0}=a^{0}\circ a^{0}=a^{0}.$$
So $L_{a}^{\rhd}(b)\in H_{a}$.  Similarly, we can obtain $L_{a^{-1}}^{\rhd}(b)\in H_{a}$. By (B2) and (B1), we have
$$L_{a}^{\rhd}L_{a^{-1}}^{\rhd}(b)=a\rhd(a^{-1}\rhd b)=a\rightharpoonup (a^{-1}\rightharpoonup b)=(a\circ a^{-1})\rightharpoonup b=a^{0}\rightharpoonup b=b^{0}\rightharpoonup b=b,$$
$$L_{a^{-1}}^{\rhd}L_{a}^{\rhd}(b)=a^{-1}\rhd(a\rhd b)=a^{-1}\rightharpoonup (a\rightharpoonup b)=(a^{-1}\circ a)\rightharpoonup b=a^{0}\rightharpoonup b=b^{0}\rightharpoonup b=b.$$
Thus $(L_{a}^{\rhd}|_{H_{a}})^{-1}=L_{a^{-1}}^{\rhd}|_{H_{a}}$. By (P1),  it follows that $L_{a}^{\rhd}|_{H_{a}}\in {\rm Aut}(H_{a},+)$.
Moreover,
\begin{equation*}
\begin{aligned}
&a^{0}+(a\rhd b)=a^{0}\circ((a^{0})^{-1}\rightharpoonup(a\rhd b))=a^{0}\circ(a^{0}\rightharpoonup(a\rightharpoonup b))\\
&=a^{0}\circ((a^{0}\circ a)\rightharpoonup b)=a^{0}\circ(a\rightharpoonup b)=a\rightharpoonup b=a\rhd b.
\end{aligned}
\end{equation*}
Therefore, $(T,+,\rhd)$ forms a strong post Clifford semigroup.
If $\psi:(T,\circ,R_{T})\rightarrow (S,\circ,R_{S})$ is a braided Clifford semigroup homomorphism, then $\psi(a\circ b)=\psi(a)\circ \psi(b)$ and $(\psi\times\psi)R_{T}=R_{S}(\psi\times\psi)$. For all $a,b\in T$, we have
$$(\psi\times\psi)R_{T}(a,b)=(\psi\times\psi)(a\rightharpoonup b,a\leftharpoonup b)=(\psi(a\rightharpoonup b),\psi(a\leftharpoonup b)),$$
$$R_{S}(\psi\times\psi)(a,b)=R_{S}(\psi(a),\psi(b))=(\psi(a)\rightharpoonup \psi(b),\psi(a)\leftharpoonup \psi(b)).$$
This implies that
$$\psi(a\rhd b)=\psi(a\rightharpoonup b)=\psi(a)\rightharpoonup \psi(b)=\psi(a)\rhd \psi(b),$$
$$\psi(a+b)=\psi(a\circ(a^{-1}\rightharpoonup b))=\psi(a)\circ(\psi(a)^{-1}\rightharpoonup \psi(b))=\psi(a)\circ \psi(b).$$
Therefore $\psi$ is a post Clifford semigroup homomorphism.
\end{proof}

\begin{theorem}
The category $\mathbb{SPCS}$ of strong post Clifford semigroups together with post Clifford semigroup homomorphisms  is isomorphic to the category $\mathbb{BCS}$ of braided Clifford semigroups together with  braided Clifford semigroup homomorphisms.
\end{theorem}
\begin{proof}
Define $\mathcal{Y}: \mathbb{SPCS} \rightarrow \mathbb{BCS}$ as follows: For any $(T,+,\rhd), (S,+,\rhd)\in {\rm obj}(\mathbb{SPCS})$ and $\psi \in {\rm Hom}((T,+,\rhd), (S,+,\rhd))$, let
$\mathcal{Y}(T,+,\rhd)=(T, \circ, R_{T})$ and $\mathcal{Y}(\psi)=\psi$. On the other hand,  define $\mathcal{P}: \mathbb{BCS}\rightarrow \mathbb{SPCS}$ as follows: For any $(T,\circ,R_{T}), (S,\circ,R_{S})\in {\rm obj}(\mathbb{BCS})$ and $\psi\in {\rm Hom}((T,\circ,R_{T}), (S,\circ,R_{S}))$, let
$\mathcal{P}(T,\circ,R_{T})=(T,+,\rhd)$ and $\mathcal{P}(\psi)=\psi$.
By Lemmas \ref{x8} and \ref{x9}, it is easy to see that $\mathcal{Y}$ and $\mathcal{P}$ are functors.

Now let $(T,+,\rhd)$ be a post Clifford semigroup and denote $\mathcal{Y}(T,+,\rhd)=(T,\circ, R_{T})$ and $\mathcal{PY}(T,+,\rhd)=(T,+',\rhd')$. Let $a,b\in T$. By Lemmas \ref{x8} and \ref{x9}, we have
$a\rhd'b=a\rightharpoonup b=a^{0}+a\rhd b=a\rhd b$ and
\begin{equation*}
\begin{aligned}
&a+' b=a\circ(a^{-1}\rightharpoonup b)=a\circ a^{0}\circ (a^{-1}\rightharpoonup b)=a\circ(a^{0}+(a^{-1}\rightharpoonup b))\\
&=a\circ((a^{-1})^{0}+(a^{-1}\rhd b))=a\circ(a^{-1}\rhd b)=a+a\rhd(a^{-1}\rhd b)\\
&=a+(a\circ a^{-1})\rhd b=a+a^{0}\rhd b=a+a^{0}+(a^{0}\rhd b)=a+b+(b\rhd a^{0}) \quad (\mbox{by  (P3)})\\
&=a+b+b^{0}+(b\rhd a)^{0}=a+b+(b+b\rhd a)^{0}=a+b+(b\circ a)^{0}\\
&=a+b+b^{0}+a^{0}=a+b.
\end{aligned}
\end{equation*}
This implies that $\rhd'=\rhd$ and $+'=+$. Thus $\mathcal{PY}={\rm id}_{\mathbb{SPCS}}$.

On the other hand, let $(T, \circ, R_{T})$ be a braided Clifford semigroup and  $\mathcal{P}(T,\circ,R_{T})=(T,+,\rhd)$ and $\mathcal{YP}(T,\circ,R_{T})=(T,\circ',R_{T}')$.  Let $a,b\in T$. By Lemmas \ref{x8} and \ref{x9}, (B2),(B9) and Lemma \ref{x3},
\begin{equation}\label{wei1}
 a\circ'b=a+(a\rhd b)=a+(a\rightharpoonup b)=a\circ(a^{-1}\rightharpoonup (a\rightharpoonup b))=a\circ(a^{0}\rightharpoonup b)=a\circ b.
\end{equation}
\begin{equation}\label{wei2}
\begin{aligned}
&R_{T}'(a,b)=(a^{0}+(a\rhd b),(a\rhd b)^{-1}\circ' a\circ' b)=(a^{0}+(a\rhd b),(a\rhd b)^{-1}\circ a\circ b)\\
&=(a^{0}+(a\rightharpoonup b),(a\rightharpoonup b)^{-1}\circ a\circ b)\\
&=(a^{0}\circ(a\rightharpoonup b), (a\rightharpoonup b)^{-1}\circ(a\rightharpoonup b)\circ(a\leftharpoonup b))\\
&=(a\rightharpoonup b,(a\rightharpoonup b)^{0}\circ(a\leftharpoonup b))=(a\rightharpoonup b,a\leftharpoonup b)=R_{T}(a,b). \quad (\mbox{by (B8)})
\end{aligned}
\end{equation}
This yields that $\circ'=\circ$ and $R_{T}'=R_{T}$.  Thus $\mathcal{YP}={\rm id}_{\mathbb{BCS}}$.
\end{proof}
In view of Lemma \ref{x8} and (\ref{wei1}) and (\ref{wei2}), we have the following result.
\begin{coro}Let $(T,\circ,R_{T})$ be a braided Clifford semigroup. Then $R_T$ is a solution of the Yang-Baxter equation on $T$.
\end{coro}
\noindent {\bf Acknowledgment:}   The authors express  their profound gratitude to Professor Li Guo at Rutgers University for his encouragement and help. The paper is supported partially  by the Nature Science Foundations of  China (12271442, 11661082).

\end{document}